\documentclass [twoside,reqno, 12pt] {amsart}

\usepackage{tikz}

\usepackage{amsfonts}
\usepackage{amssymb}
\usepackage{a4}
\usepackage{color}

\newtheorem{thm}{Theorem}[section]

\newtheorem{lem}[thm]{Lemma}
\newtheorem{prop}[thm]{Proposition}

\newtheorem{rem}[thm]{Remark}

\theoremstyle{definition}

\numberwithin{equation}{section}

\renewcommand{\Re}{\hbox{Re}\,}
\renewcommand{\Im}{\hbox{Im}\,}

\newcommand{\C}{\mathbb{C}}

\newcommand{\R}{\mathbb{R}}

\newcommand{\supp}{\operatorname{supp}}

\newcommand{\Z}{\mathbb{Z}}

\def\hat{\widehat}
\def\tilde{\widetilde}
\def \bfo {\begin {eqnarray*} }
\def \efo {\end {eqnarray*} }
\def \ba {\begin {eqnarray*} }
\def \ea {\end {eqnarray*} }
\def \beq {\begin {eqnarray}}
\def \eeq {\end {eqnarray}}
\def \supp {\hbox{supp }}

\def \dist {\hbox{dist}}

\def \det {\hbox{det}}

\def \p {\partial}

\def\hat{\widehat}
\def\tilde{\widetilde}
\def \bfo {\begin {eqnarray*} }
\def \efo {\end {eqnarray*} }
\def \ba {\begin {eqnarray*} }
\def \ea {\end {eqnarray*} }
\def \beq {\begin {eqnarray}}
\def \eeq {\end {eqnarray}}
\def \supp {\hbox{supp }}

\def \dist {\hbox{dist}}

\def \det {\hbox{det}}

\def \p {\partial}

\begin{document}

 \title[inverse problems for polyharmonic operators]{Inverse boundary problems for polyharmonic operators with unbounded potentials}

\author[Krupchyk]{Katsiaryna Krupchyk}

\address
        {K. Krupchyk, Department of Mathematics\\
University of California, Irvine\\ 
CA 92697-3875, USA }

\email{katya.krupchyk@uci.edu}

\author[Uhlmann]{Gunther Uhlmann}

\address
       {G. Uhlmann, Department of Mathematics\\
       University of Washington\\
       Seattle, WA  98195-4350\\
       USA\\
       Department of Mathematics and Statistics \\
       University of Helsinki\\
         P.O. Box 68 \\
         FI-00014   Helsinki\\
         Finland}
\email{gunther@math.washington.edu}

\maketitle

\begin{abstract}
We show that the knowledge of the Dirichlet--to--Neumann map on the boundary of a bounded open set in $\R^n$ for the 
perturbed polyharmonic operator $(-\Delta)^m +q$ with $q\in L^{\frac{n}{2m}}$, $n>2m$, determines the potential $q$ in the set uniquely. In the course of the proof, we construct a special Green function for the polyharmonic operator and establish its mapping properties in suitable weighted $L^2$ and $L^p$ spaces. The $L^p$ estimates for  the special Green function are derived from $L^p$ Carleman estimates with linear weights for the polyharmonic operator.

\end{abstract}

\section{Introduction}

Let $\Omega\subset\R^n$ be a bounded open set with $C^\infty$ boundary, and let $(-\Delta)^m$, $m=1,2,\dots,$ be a polyharmonic operator. Let $q\in L^{\frac{n}{2m}}(\Omega)$ be a complex valued potential.  We shall assume throughout the paper that $n>2m$.  

 Let $\gamma$ be the Dirichlet trace operator, given by
\[
\gamma:H^{m}(\Omega)\to  \prod_{j=0}^{m-1}H^{m-j-1/2}(\p \Omega),\quad \gamma u=(u|_{\p\Omega},\p_{\nu}u|_{\p \Omega},\dots,\p_{\nu}^{m-1}u|_{\p \Omega}),
\]
which is bounded and surjective, see \cite[Theorem 9.5, p. 226]{Grubbbook2009}. Here and in what follows $H^s(\Omega)$ and $H^s(\p \Omega)$, $s\in \R$, are the standard $L^2$--based Sobolev spaces in $\Omega$ and its boundary $\p\Omega$, respectively, and  $\nu$ is the exterior unit normal to the boundary.  We shall also set 
\[
H^m_0(\Omega)=\{u\in H^m(\Omega): \gamma u=0\}.
\]
An application of the Sobolev embedding theorem shows that the operator of multiplication by $q$ is continuous:  $H^m_0(\Omega)\to H^{-m}(\Omega)$, and 
standard arguments, see Appendix \ref{sec_appendix}, imply that the operator 
\begin{equation}
\label{eq_int_1}
(-\Delta)^m+q: H^m_0(\Omega)\to H^{-m}(\Omega)=(H_0^m(\Omega))'
\end{equation}
is Fredholm  of index zero. Furthermore, the  operator in \eqref{eq_int_1}  has a discrete spectrum. 

We shall assume throughout the paper that 
\begin{itemize}
\item[\textbf{(A)}] $0$ is not in the spectrum of the operator \eqref{eq_int_1}.
\end{itemize}
It follows that  for $f=(f_0,\dots, f_{m-1})\in\prod_{j=0}^{m-1} H^{m-j -1/2}(\p \Omega)$, the Dirichlet problem, 
\begin{equation}
\label{eq_int_2}
\begin{aligned}
((-\Delta)^m+q)u&=0\quad\textrm{in}\quad \Omega,\\
\gamma u&=f \quad\textrm{on}\quad \p\Omega,
\end{aligned}
\end{equation}
has a unique solution $u\in H^m(\Omega)$. We introduce the Dirichlet--to--Neumann map by
\[
\langle \Lambda_q f, \overline{h} \rangle=\sum_{|\alpha|=m}  \frac{m!}{\alpha!} \int_\Omega  D^\alpha u \overline{D^\alpha v} dx + \int_\Omega u q\overline{v} dx,
\]
where 
$h\in \prod_{j=0}^{m-1}H^{m-j-1/2}(\p \Omega)$, and $v\in H^m(\Omega)$ is such that $\gamma v=h$.  
It is shown in Appendix \ref{sec_appendix} that $\Lambda_q$ is well defined and 
\[
\Lambda_q: \prod_{j=0}^{m-1} H^{m-j -1/2}(\p \Omega)\to \bigg(\prod_{j=0}^{m-1} H^{m-j -1/2}(\p \Omega)\bigg)'=\prod_{j=0}^{m-1} H^{-m+j +1/2}(\p \Omega)
\]
is a linear continuous map.  Notice that when $m=1$, we recover the standard Dirichlet--to--Neumann map for the Schr\"odinger equation, given by
\begin{equation}
\label{eq_int_DN_map_Schr}
H^{1/2}(\p \Omega)\ni f\mapsto  \p_\nu u|_{\p\Omega}\in H^{-1/2}(\p \Omega).
\end{equation}

The inverse boundary value problem for the perturbed polyharmonic operator \eqref{eq_int_1} is to determine a potential $q$ in $\Omega$ from the knowledge of the Dirichlet--to--Neumann map $\Lambda_q$. 

This problem has been studied extensively in the case of the Schr\"odinger operator, i.e.  when $m=1$. It was shown in \cite{Sylvester_Uhlmann_1987}  that a potential $q\in L^\infty(\Omega)$ is uniquely  determined in $\Omega$ from the knowledge of the Dirichlet--to--Neumann map \eqref{eq_int_DN_map_Schr} for the Schr\"odinger equation.  
The proof of this result in \cite{Sylvester_Uhlmann_1987} is based on a construction of complex geometric optics solutions to the Schr\"odinger equation, with an $L^2$ control of the remainder. The proof also goes through for some unbounded potentials, e.g. $q\in L^{n}(\Omega)$.  
In \cite{Lavine_Nachman}  a global uniqueness result was established for $q\in L^{n/2}(\Omega)$, following an earlier result of  \cite{Chanillo_1990}  for $q\in L^{n/2+\varepsilon}(\Omega)$, $\varepsilon >0$.  It turns out that $L^2$  methods are no longer sufficient in the proofs in \cite{Chanillo_1990}  and \cite{Lavine_Nachman}, and it becomes essential to control the remainders of complex geometric optics solutions in suitable $L^p$ spaces.  Inverse boundary value problems for the Schr\"odinger equation on certain classes of manifolds were studied in  \cite{Dos_Santos_F_Kenig_Salo_Uhlmann} and \cite{Dos_Santos_F_Kenig_Salo_unbound}, in the case of  $L^\infty$  and $L^{n/2}$ potentials, respectively.

Turning our attention to the case of polyharmonic operators, let us remark that the areas of physics and geometry where such operators occur, include  the
study of the Kirchhoff plate equation in the theory of elasticity, and the study of the Paneitz-Branson operator in conformal geometry, see \cite{Gazzola_book}.   It was shown in \cite{Isakov_1991} that  a potential $q\in L^\infty(\Omega)$ can be uniquely recovered from the knowledge of the Dirichlet--to--Neumann map $\Lambda_q$ for the perturbed biharmonic equation, i.e. $m=2$.  In \cite{Ikehata_1991} an alternative approach to this problem  was developed and the uniqueness result was extended to $q\in  L^{n/2}(\Omega)$, $n>4$.  The proofs in \cite{Isakov_1991} and \cite{Ikehata_1991}  rely upon $L^2$  methods only.  Inverse spectral problems for a potential perturbation of the polyharmonic operator were studied in \cite{Krupchyk_Paivarinta},  and  inverse boundary value problems for a first order perturbation of the polyharmonic operator were addressed in \cite{Krupchyk_Lassas_Uhlmann_Funct, Krupchyk_Lassas_Uhlmann_Trans}, again using $L^2$  techniques.

The purpose of this paper is to study the problem of recovering a potential $q\in L^{\frac{n}{2m}}(\Omega)$, $n>2m$, from the Dirichlet--to--Neuman map $\Lambda_q$, associated to the perturbed polyharmonic operator $(-\Delta)^m+q$. This can be viewed as an extension of the study of \cite{Lavine_Nachman} for the Schr\"odinger equation to  the case of the polyharmonic equation.  

The assumption $q\in L^{\frac{n}{2m}}(\Omega)$, $n>2m$, seems natural as it guarantees that the strong unique continuation principle holds for the operator $(-\Delta)^m+q$, see \cite{Jerison_Kenig_1985}  in the case $m=1$, and  \cite{Laba_1988}  in the case $m\ge 2$.  Furthermore, it is known that  this condition is optimal in the class of $L^p$ potentials, see \cite{Jerison_Kenig_1985}. 

While the operator of multiplication by  $q\in L^\infty(\Omega)$ is continuous: $H^m(\Omega)\to L^2(\Omega)$, taking  $q\in L^{\frac{n}{2m}}(\Omega)$, we get a continuous operator: $H^m(\Omega)\to L^{\frac{2n}{n+2m}}(\Omega)$. Here  $  L^2(\Omega)\subset L^{\frac{2n}{n+2m}}(\Omega) $, and therefore, when constructing complex geometric optics solutions for the perturbed polyharmonic operator, it becomes crucial to control the remainders not only in $L^2(\Omega)$ but also in $L^{\frac{2n}{n-2m}}(\Omega)$, which is  the dual space of  $L^{\frac{2n}{n+2m}}(\Omega)$.

The following uniqueness result is the main result of this paper. 
\begin{thm}
\label{thm_main}
Let $q_1,q_2\in L^{\frac{n}{2m}}(\Omega)$, $n>2m$, and assume that \emph{(A)} holds for $(-\Delta)^m+q_j$, $j=1,2$.  If $\Lambda_{q_1}=\Lambda_{q_2}$, then $q_1=q_2$ in $\Omega$. 
\end{thm}

In the case $m=1$ this result is due to \cite{Lavine_Nachman}.  In the proof of Theorem \ref{thm_main} we shall follow the method of \cite{Lavine_Nachman}, which we learned from the work \cite{Dos_Santos_F_Kenig_Salo_unbound}.

The crucial role in \cite{Lavine_Nachman}, following \cite{Sylvester_Uhlmann_1987} and \cite{Faddeev_1965},  is played by the  fundamental solution  
\[
g_\zeta^{(1)}=\mathcal{F}^{-1} \bigg( \frac{1}{p_\zeta (\xi)}\bigg)\in \mathcal{S}'(\R^n).
\] 
of the conjugated Laplacian $e^{-x\cdot\zeta} (-\Delta) e^{x\cdot\zeta}=-\Delta-2\zeta\cdot \nabla$ with $\zeta\in \C^n\setminus\{0\}$, $\zeta\cdot\zeta=0$. 
 Here $p_\zeta(\xi)=|\xi|^2-2i\zeta\cdot\xi$ is the symbol of the operator, and $\mathcal{F}$ is the Fourier transformation.  
 The argument of  \cite{Lavine_Nachman} consists of two main ingredients.  The first one is the use of two fundamental estimates for the convolution operator $G_\zeta^{(1)}f=g_\zeta^{(1)}*f$, in suitable weighted $L^2$ and $L^p$ spaces. Such estimates have been established in \cite{Sylvester_Uhlmann_1987} and  \cite{Kenig_Ruiz_Sogge_1987}, respectively, see also \cite{Chanillo_1990}.  The second ingredient is an approximation of an $L^p$ function by  a sequence of $L^\infty$ functions.

To follow the method of  \cite{Lavine_Nachman}, a tempered fundamental solution of the conjugated polyharmonic operator $e^{-x\cdot\zeta} (-\Delta)^m e^{x\cdot\zeta}=(-\Delta-2\zeta\cdot \nabla)^m$ with $\zeta\in \C^n\setminus\{0\}$, $\zeta\cdot\zeta=0$, with good mapping properties of the corresponding convolution operator in appropriate weighted $L^2$ and $L^p$ spaces, should be available.

A new difficulty in the construction of such a fundamental solution, compared with the case of the Laplacian, arises since, while $1/p_\zeta(\xi)\in \mathcal{S}'(\R^n)\cap L^1_{\text{loc}}(\R^n)$, we have $1/(p_\zeta(\xi))^m\notin L^1_{\text{loc}}(\R^n)$ for $m\ge 2$, and therefore, it cannot be viewed as a distribution directly.  Here we find a way to regularize $1/(p_\zeta(\xi))^m$ and  obtain the following result, where $L^2_\sigma (\R^n)$, $\sigma\in \R$, stands  for the weighted $L^2$ space with the norm
\begin{equation}
\label{eq_int_weighted_L_2}
\|f\|_{L^2_\sigma}=\bigg( \int_{\R^n} (1+|x|^2)^\sigma |f(x)|^2dx\bigg)^{1/2}.
\end{equation}
\begin{thm}
\label{thm_main_2}
Let $m=1, 2, 3,\dots$, and let $\zeta\in \C^n\setminus\{0\}$ be such that $\zeta\cdot\zeta=0$. There exists  $g_\zeta^{(m)}\in \mathcal{S}'(\R^n)$  such that 
\[
(-\Delta-2\zeta\cdot \nabla)^m g_\zeta^{(m)}=\delta,
\]  
and such that the operator $G_\zeta^{(m)}: \mathcal{S}(\R^n)\to \mathcal{S}'(\R^n)$,  defined by
\[
G_\zeta^{(m)} f=g_\zeta^{(m)}*f, 
\]
enjoys the following properties: 
\begin{itemize}
\item[(i)] 
For $|\zeta|\ge s_0>0$, the operator
 $G^{(m)}_{\zeta}$
can be extended to a bounded operator $G_\zeta^{(m)}: L^2_{\sigma+2m-1}(\R^n)\to L^2_{\sigma}(\R^n)$, $-m<\sigma<1-m$, such that
\begin{equation}
\label{eq_thm_2_est_1}
\|G^{(m)}_\zeta f\|_{L^2_\sigma}\le \frac{C}{|\zeta|^m}\|f\|_{L^2_{\sigma+2m-1}}.
\end{equation}
\item[(ii)]
$G_\zeta^{(m)}$ extends  to a bounded operator $G_\zeta^{(m)}: L^{\frac{2n}{n+2m}}(\R^n)\to L^{\frac{2n}{n-2m}}(\R^n)$ with the bound
\begin{equation}
\label{eq_thm_2_est_2}
\|G_\zeta^{(m)}f\|_{L^{\frac{2n}{n-2m}}}\le C\|f\|_{L^{\frac{2n}{n+2m}}},
\end{equation}
uniformly in $\zeta$. When $f\in L^{\frac{2n}{n+2m}}(\R^n)$,  the function $w=G_\zeta^{(m)}f$  is the unique $L^{\frac{2n}{n-2m}}(\R^n)$ -- solution of the equation
$
(-\Delta-2\zeta\cdot \nabla)^m w= f$.
\end{itemize}
\end{thm}

When $m=1$, this result is known thanks to the works  \cite{Sylvester_Uhlmann_1987} ,  \cite{Kenig_Ruiz_Sogge_1987}, and  \cite{Chanillo_1990}.

A fundamental result of  \cite{Hormander_1958} and \cite{Lojasiewicz_1959} establishes the existence of a  tempered fundamental solution for a general partial differential operator with constant coefficients. As our applications require estimates for the corresponding convolution operators in (weighted) Lebesgue spaces,  in the proof of Theorem \ref{thm_main_2} we proceed much more concretely and construct the distribution $g_\zeta^{(m)}$ essentially explicitly.    
  In the case of the  biharmonic operator, a tempered fundamental solution of $(-\Delta-2\zeta\cdot\nabla)^2$  is constructed and weighted $L^2$ estimates  are obtained in \cite{Ikehata_1991}.  Our construction is different and works for a general polyharmonic operator. 
We should also mention that in \cite{Isakov_1991} a regular (non tempered) fundamental solution of $(-\Delta-2\zeta\cdot\nabla)^2$ is used, see \cite[Theorem 10.2.1]{Hormander_book_2}, and local $L^2$ estimates are obtained. It seems that local $L^p$ estimates are not easily obtained for the corresponding convolution operator, and therefore, this approach is not pursued in this work.   

An important ingredient in the proof of  \eqref{eq_thm_2_est_1} is the weighted $L^2$ estimate for the operator $G_\zeta^{(1)}$, obtained in \cite{Sylvester_Uhlmann_1987}.  To prove the estimate \eqref{eq_thm_2_est_2} we use uniform $L^p$ Carleman estimates with linear weights for the polyharmonic operator. Since we did not find any reference for such estimates,  in Section \ref{sec_Carleman_estimates}  we follow \cite{Wolff_1993} and derive them as a consequence of the corresponding $L^p$ Carleman estimates with logarithmic weights for the polyharmonic operator, established in \cite{Jerison_Kenig_1985}.

The paper is organized as follows.  In Section \ref{sec_Carleman_estimates}  we discuss  $L^p$ Carleman estimates with linear weights for the polyharmonic operator.  Section \ref{sec_Green_function} is devoted to the construction of a tempered fundamental solution to the conjugated polyharmonic operator $e^{-x\cdot\zeta} (-\Delta)^m e^{x\cdot\zeta}$ and to the proof of Theorem \ref{thm_main_2}.  
Section \ref{sec_cgo} contains the construction of complex geometric optics solutions  to the perturbed polyharmonic equation  with a potential $q\in L^{\frac{n}{2m}}(\Omega)$.  The proof of Theorem \ref{thm_main} is then concluded in Section \ref{sec_proof_thm_main}.   Appendix \ref{sec_appendix} is concerned with the wellposedness of the Dirichlet problem for $(-\Delta)^m+q$ with $q\in L^{\frac{n}{2m}}(\Omega)$, and is included for the completeness of the presentation.

\section{$L^p$ Carleman estimates with linear weights for polyharmonic operators}

\label{sec_Carleman_estimates}

The purpose of this section is to present $L^p$ Carleman estimates with linear weights for polyharmonic operators, which generalize the corresponding estimates of  \cite{Kenig_Ruiz_Sogge_1987}, obtained for the Laplacian.

In the work  \cite{Wolff_1993}  it is explained how to obtain the $L^p$ Carleman estimates with linear weights for the Laplacian of \cite{Kenig_Ruiz_Sogge_1987} from the $L^p$ Carleman estimates with logarithmic weights, established in \cite{Jerison_Kenig_1985}.  The work  \cite{Jerison_Kenig_1985} contains also  $L^p$ Carleman estimates with logarithmic weights for polyharmonic operators, and  following \cite{Wolff_1993}, we shall  derive $L^p$ Carleman estimates with linear weights for polyharmonic operators from these  estimates. 

Let us start by formulating the following result of \cite{Jerison_Kenig_1985}.

\begin{thm} [\cite{Jerison_Kenig_1985}]
Let $m$ be an integer, $1\le m<n/2$,  and suppose that $p=\frac{2n}{n+2m}$ and $q=\frac{2n}{n-2m}$ (i.e. $\frac{1}{p}+\frac{1}{q}=1$ and $\frac{1}{p}-\frac{1}{q}=\frac{2m}{n}$).  Let  $t>\frac{n}{q}$ and $\delta=\text{dist}(t- \frac{n}{q},\Z)>0$. Then there exists a constant $C=C(\delta,n)$, depending only on $\delta$ and $n$, such that  for every $u\in C^\infty_0(\R^n\setminus\{0\})$,
\begin{equation}
\label{eq_JK}
\||x|^{-t} u\|_{L^{q}(\R^n)}\le C\| |x|^{-t} (-\Delta)^m u\|_{L^{p}(\R^n)}.
\end{equation}
\end{thm}

In the next result we establish the $L^p$ Carleman estimates with linear weights for the polyharmonic operator. 
\begin{prop}
Let $m$ be an integer, $1\le m<n/2$, and $p=\frac{2n}{n+2m}$ and $q=\frac{2n}{n-2m}$. Then 
\begin{equation}
\label{eq_KRS}
\|e^{k\cdot x} u\|_{L^{q}(\R^n)}\le C\|e^{k\cdot x} (-\Delta)^m u\|_{L^p(\R^n)}
\end{equation}
uniformly in $k\in \R^n$ and $u\in C_0^\infty(\R^n)$.
\end{prop}

\begin{proof}
First notice that when $k=0$, the inequality \eqref{eq_KRS} follows from the Hardy--Littlewood--Sobolev inequality, see \cite[Theorem 4.5.3]{Hormander_book_1}. 

Let now $k\in \R^n\setminus\{0\}$ be fixed and let $u\in C^\infty_0(\R^n)$.  Consider the function $\tilde u(x)=u(x+tk)$. Since for $t>0$ sufficiently large, $\tilde u=0$ near zero, applying \eqref{eq_JK} to $\tilde u$, we get 
\begin{equation}
\label{eq_2_1}
\||x|^{-t} \tilde u\|_{L^{q}(\R^n)}\le C\| |x|^{-t} (-\Delta)^m \tilde u\|_{L^{p}(\R^n)},
\end{equation}
for $t>0$ sufficiently large such that $\delta=\text{dist}(t- \frac{n}{q},\Z)>0$.
Making the change of variables $x+tk\mapsto x$ in \eqref{eq_2_1} , we obtain that
\begin{equation}
\label{eq_2_2}
\||x-tk|^{-t}  u\|_{L^{q}(\R^n)}\le C\| |x-tk|^{-t} (-\Delta)^m u\|_{L^{p}(\R^n)}.
\end{equation}
Writing 
\[
|x-tk|=t|k|\sqrt{1-2\frac{x\cdot k}{t|k|^2}+\frac{|x|^2}{t^2 |k|^2}},
\]
we see that \eqref{eq_2_2} is equivalent to
\begin{equation}
\label{eq_2_3}
\| \bigg( 1-2\frac{x\cdot k}{t|k|^2}+\frac{|x|^2}{t^2 |k|^2} \bigg)^{-t/2}  u\|_{L^{q}(\R^n)}\le C\| \bigg( 1-2\frac{x\cdot k}{t|k|^2}+\frac{|x|^2}{t^2 |k|^2} \bigg)^{-t/2} (-\Delta)^m u\|_{L^{p}(\R^n)}.
\end{equation}
We have
\begin{align*}
\bigg( 1-2\frac{x\cdot k}{t|k|^2}+\frac{|x|^2}{t^2 |k|^2} \bigg)^{-t/2}&=\exp\bigg(-\frac{t}{2}\log(1-2\frac{x\cdot k}{t|k|^2}+\frac{|x|^2}{t^2 |k|^2} )\bigg)\\
&= \exp\bigg(\frac{x\cdot k}{|k|^2}+\mathcal{O}_{k,\text{supp}\,u}(\frac{1}{t})\bigg)\to e^{\frac{x\cdot k}{|k|^2}},
\end{align*}
uniformly on $\supp u$, as $t\to \infty$ away from the set $\Z+\frac{n}{q}$.
Hence, passing to the limit in \eqref{eq_2_3},  we get 
\begin{equation}
\label{eq_sec_2KRS}
\| e^{\frac{x\cdot k}{|k|^2}} u\|_{L^{q}(\R^n)}\le C\| e^{\frac{x\cdot k}{|k|^2}} (-\Delta)^m u\|_{L^{p}(\R^n)},
\end{equation}
with the same constant as in \eqref{eq_JK}. 
Replacing $k\in \R^n\setminus\{0\}$ in \eqref{eq_sec_2KRS} by $k/|k|^2$, we recover \eqref{eq_KRS}. The proof is complete. 
\end{proof}

\section{A special Green function for the polyharmonic operator. Proof of Theorem \ref{thm_main_2}}

\label{sec_Green_function}

Let $\zeta\in\C^n\setminus\{0\}$ be such that $\zeta\cdot\zeta=0$, and let us consider the constant coefficient differential operator, 
\begin{equation}
\label{eq_1_0}
e^{-x\cdot \zeta} (-\Delta)^m e^{x\cdot \zeta}=(-\Delta-2\zeta\cdot \nabla)^m.
\end{equation}
The purpose of this section is to construct a suitable tempered fundamental solution of the operator \eqref{eq_1_0}, and to prove Theorem \ref{thm_main_2}.  
To that end let us consider the equation,
\begin{equation}
\label{eq_1_1}
(-\Delta-2\zeta\cdot \nabla)^m w=\delta \quad \text{in}\quad \R^n,
\end{equation}
where $w\in\mathcal{S}'(\R^n)$. 
Taking the Fourier transform in \eqref{eq_1_1}, we obtain that 
\[
(|\xi|^2-2i\zeta\cdot\xi)^m\hat w(\xi)=1.
\]
Here and in what follows we use the  normalization,
\[
\hat f(\xi)=\mathcal{F}(f)(\xi)=\int_{\R^n} e^{-ix\cdot\xi} f(x)dx, \quad f\in \mathcal{S}(\R^n), 
\]
extended to the space $\mathcal{S}'(\R^n)$ in the usual way. 

The condition $\zeta\cdot\zeta=0$ is equivalent to the fact that 
\begin{equation}
\label{eq_1_2}
\Re\zeta\cdot\Im\zeta=0,\quad |\Re\zeta|=|\Im\zeta|.
\end{equation}
Using \eqref{eq_1_2} together with the fact that the Laplacian commutes with rotations,  we  may and shall assume, without loss of generality, that $\zeta$ in \eqref{eq_1_1} has the form,
\[
\zeta=s e_1-is e_2, \quad s=|\zeta|/\sqrt{2}>0,
\] 
where $e_1$ and $e_2$ are the first two vectors in the standard orthonormal basis in $\R^n$. 

The symbol of the operator $(-\Delta-2\zeta\cdot \nabla)^m$ is equal to  $(p_\zeta(\xi))^m$, where
\[
p_\zeta(\xi)=|\xi|^2-2i\zeta\cdot\xi=|\xi-s e_2|^2-s^2 -2is\xi_1.
\]
The characteristic set of $(-\Delta-2\zeta\cdot \nabla)^m$ is given by
\begin{equation}
\label{eq_5_sigma_zeta}
\Sigma_\zeta=\{\xi\in\R^n: p_\zeta(\xi)=0\}=\{\xi\in\R^n: \xi_1=0, |\xi-s e_2|=s\}.
\end{equation}
Thus, $\Sigma_\zeta$ is the codimension $2$ sphere, which is obtained as the intersection of the hyperplane $\xi_1=0$
and the $(n-1)$--dimensional sphere, centered at $se_2$ and of radius $s$. 

In what follows, we shall use the standard notation $a \lesssim b$ for $a, b>0$, which means that there exists a constant $C>0$ independent of $\zeta$ such that $a\le C b$. We shall also write $a\sim b$ when $a\lesssim b$ and $b\lesssim a$.  When the constant $C$ depends on a large parameter $M$, we indicate this dependence by writing $a\sim_M b$.

We shall need the following result concerning the behavior of  $p_\zeta(\xi)$ in $\R^n$, established in \cite{Sylvester_Uhlmann_1987}, see also \cite{Haberman_Tataru}.   We shall present the proof  since some of the estimates in the proof will be used in what follows. 
\begin{lem}
\label{lem_behavior_p_zeta}
For $|\xi|\ge 4|\zeta|$, we have  
\begin{equation}
\label{eq_1_3_0}
|p_\zeta(\xi)|\sim |\xi|^2.
\end{equation}
For $|\xi|\le M |\zeta|$ with a fixed constant $M$, 
\begin{equation}
\label{eq_1_3}
|p_\zeta (\xi)|\sim_M s d(\xi,\Sigma_\zeta),
\end{equation}
 where $d(\xi,\Sigma_\zeta)$ is the distance from $\xi$ to $\Sigma_\zeta$.
\end{lem}

\begin{proof}
Assume first that $|\xi|\ge 4|\zeta|$. Then  \eqref{eq_1_3_0} follows from the following estimates, 
\[
|p_\zeta(\xi)|\le |\xi|^2+2|\zeta||\xi|\le \frac{3}{2}|\xi|^2,
\]
and
\[
|p_\zeta(\xi)|\ge |\xi|^2-2|\zeta||\xi|\ge \frac{1}{2}|\xi|^2.
\]

To see \eqref{eq_1_3}, we first observe that for $|\xi|\le M |\zeta|$,
\begin{equation}
\label{eq_1_3_1}
|p_\zeta(\xi)|\sim ||\xi-se_2|^2-s^2|+2s|\xi_1|\sim_M s(||\xi-s e_2|-s|+|\xi_1|).
\end{equation}
Let $\xi\in \R^n$. Then for any $\eta\in \Sigma_\zeta$, we get
\begin{equation}
\label{eq_1_4_0}
\begin{aligned}
&|\xi-\eta|\ge |\xi_1|,\\
&|\xi-\eta|\ge ||\xi-s e_2|-|\eta-s e_2|| = ||\xi-s e_2|-s|,
\end{aligned}
\end{equation}
and therefore,
\begin{equation}
\label{eq_1_3_2}
d(\xi,\Sigma_\zeta)\ge \frac{1}{2}(||\xi-s e_2|-s|+|\xi_1|). 
\end{equation}
On the other hand, we have 
\begin{equation}
\label{eq_1_4}
d(\xi,\Sigma_\zeta)\sim |\xi_1|+\inf_{\eta':  |\eta'-s e_2|=s} |\xi'-\eta'|= |\xi_1| +||\xi'-se_2|-s|.
\end{equation}
Here we have used the fact that the distance from the point $\xi'=(\xi_2,\dots, \xi_n)\in \R^{n-1}$ to the sphere $\{\eta'=(\eta_2,\dots, \eta_n)\in \R^{n-1}:  |\eta'-s e_2|=s\}$ is given by $||\xi'-se_2|-s|$.

If $|\xi'-se_2|-s\ge 0$, then 
\[
||\xi'-se_2|-s|=|\xi'-se_2|-s\le |\xi-se_2|-s= ||\xi-se_2|-s|.
\]
If $|\xi'-se_2|-s<0$, then 
\[
||\xi'-se_2|-s|=s-|\xi'-se_2|\le |\xi_1|+s-|\xi-se_2|\le |\xi_1|+|s-|\xi-se_2||. 
\]
Thus, 
\begin{equation}
\label{eq_1_3_3}
d(\xi,\Sigma_\zeta)\lesssim ||\xi-s e_2|-s|+|\xi_1|,
\end{equation}
uniformly in $s>0$. 
Hence, using \eqref{eq_1_3_1}, \eqref{eq_1_3_2} and  \eqref{eq_1_3_3}, we obtain \eqref{eq_1_3}. The proof is complete.  
\end{proof}

The next result is well-known and  is presented here for the convenience of the reader. 
\begin{lem}
\label{lem_tempered}
Let $n\ge 3$. 
Then for every $\zeta\in \C^n$, $\zeta\cdot\zeta=0$, and $|\zeta|\ge 4$,   we have 
\[
 \frac{1}{p_\zeta(\xi)}\in L^q_{\emph{\text{loc}}}(\R^n),
\]
if and only if $1\le q<2$, and for such $q$, 
\begin{equation}
\label{eq_dist_l_q_l_p}
\frac{1}{p_\zeta(\xi)}\in L^q(\R^n)+\langle \xi \rangle^{-2} L^\infty(\R^n) \subset \mathcal{S}'(\R^n).
\end{equation}
\end{lem}

\begin{proof}
The complex vector $\zeta$ will be kept fixed in the proof.  Let $K\subset\R^n$ be a fixed compact set. Then by \eqref{eq_1_3}, for $\xi\in K$, we have $|p_\zeta(\xi)|\sim  d(\xi,\Sigma_\zeta)$.
Let $q\ge 1$ and let us write 
\[
\int_K \frac{1}{|p_\zeta(\xi)|^q}d\xi\sim \int_{\{\xi\in K:d(\xi, \Sigma_\zeta)\le 1\}} \frac{1}{(d(\xi,\Sigma_\zeta))^q}d\xi+\int_{\{\xi\in K:d(\xi, \Sigma_\zeta)\ge 1\}} \frac{1}{( d(\xi,\Sigma_\zeta))^q}d\xi,
\]
where the last integral is finite. Recalling from \eqref{eq_1_4} that 
\[
d(\xi,\Sigma_\zeta)\sim |\xi_1| +||\xi'-se_2|-s|,
\]
and passing to the polar coordinates in $\xi'$, centered at $s e_2$, i.e. $\xi'=s e_2+r\theta$, $r>0$, $\theta\in \mathbb{S}^{n-2}$, we get 
\begin{align*}
&\int_{\{\xi\in K:d(\xi, \Sigma_\zeta)\le 1\}} \frac{1}{d(\xi,\Sigma_\zeta)^q}d\xi\sim \int_{\mathbb{S}^{n-2}}\int_{s-1}^{s+1}\int_{|\xi_1|\le 1}
\frac{r^{n-2}}{(|\xi_1|+|r-s|)^q} d\xi_1dr d\theta\\
&\sim
\int_{s-1}^{s+1}\int_{|\xi_1|\le 1}
\frac{1}{(|\xi_1|+|r-s|)^q} d\xi_1dr=
\int_{|\xi_2|\le 1}\int_{|\xi_1|\le 1} \frac{1}{(|\xi_1|+|\xi_2|)^q} d\xi_1d\xi_2.
\end{align*}
Here we have used the fact that $s-1\le r\le s+1$ and $s>2$ is fixed.  The latter integral is finite precisely when $q<2$. 

To check \eqref{eq_dist_l_q_l_p}  it suffices to consider the decomposition, 
\[
\frac{1}{p_\zeta(\xi)}= \chi(\xi)\frac{1}{p_\zeta(\xi)}+(1-\chi(\xi)) \frac{1}{p_\zeta(\xi)}.
\]
Here $\chi\in C_0^\infty(\R^n)$ is such that  $\chi=1$ on $\{\xi: |\xi|<4|\zeta|\}$.
The proof is complete. 
\end{proof}

In what follows we shall consider convolutions of tempered distributions and  Schwartz functions.  Let us recall briefly the standards facts about them,  see \cite{Hormander_book_1}. 
Let $u\in \mathcal{S}'(\R^n)$ and  $f\in\mathcal{S}(\R^n)$. Then the convolution $u*f$ is defined by
\[
(u*f)(x)=u_y(f(x-y)).
\]
We have 
\[
u*f\in \mathcal{S}'(\R^n)\cap C^\infty(\R^n),
\]
and 
\[
\mathcal{F}(u*f)= \hat f \hat u\in \mathcal{S}'(\R^n). 
\]

When $m=1$, the distribution 
\[
g^{(1)}_\zeta(x)=\mathcal{F}^{-1}\bigg(\frac{1}{p_\zeta(\xi)}\bigg)\in\mathcal{S}'(\R^n)
\]
is a tempered fundamental solution of the operator $-\Delta-2\zeta\cdot \nabla$. 
This fundamental solution was introduced in \cite{Faddeev_1965} and \cite{Sylvester_Uhlmann_1987}. 
The convolution  operator 
\[
G_\zeta^{(1)}f:=g_\zeta^{(1)}*f:\mathcal{S}(\R^n)\to \mathcal{S}'(\R^n)
\]
is continuous, and   in \cite{Sylvester_Uhlmann_1987}, it was shown that for $|\zeta|\ge s_0>0$, we have  $G_\zeta^{(1)}:L^2_{\sigma+1}(\R^n)\to L^2_\sigma(\R^n)$, $-1<\sigma<0$, with the bound
\[
\|G_\zeta^{(1)}f\|_{L^2_\sigma(\R^n)}\le \frac{C}{|\zeta|}\|f\|_{L^2_{\sigma+1}(\R^n)}.
\]

When $m\ge 2$, according to Lemma \ref{lem_tempered} we have $\frac{1}{(p_\zeta(\xi))^m}\notin L^1_{\text{loc}}(\R^n)$, and therefore, it cannot be viewed as a distribution directly.  Hence, we shall proceed to regularize it.  Following \cite{Sylvester_Uhlmann_1987} let us  introduce an open cover of $\R^n$ in the following way. 
Here it will be convenient  to denote by $\Sigma(s)$ the characteristic set $\Sigma_\zeta$, given by 
\eqref{eq_5_sigma_zeta}.  
Let us set
\[
V_1(s)=\R^n\setminus N_{\frac{s}{2n}}(\Sigma(s)). 
\]
where 
\[
N_r(\Sigma(s))=\{\xi\in\R^n: d(\xi,\Sigma(s))\le r\},  \quad r>0.
\] 
To construct an open cover of the compact set $N_{\frac{s}{2n}}(\Sigma(s))$, we observe from  \eqref{eq_1_4_0} that if $\xi\in N_{\frac{s}{2n}}(\Sigma(s))$ then $|\xi_1|\le \frac{s}{2n}$ and $|\xi-s e_2|\ge s-\frac{s}{2n}$.  Therefore, the length of at least one component of $\xi-s e_2$ must be greater or equal to  $\frac{1}{\sqrt{n}}(s-\frac{s}{2n})>\frac{s}{2n}$, $n\ge 2$. Thus, letting
\begin{equation}
\label{eq_5_sets_V_j}
\begin{aligned}
V_2(s)&=\{\xi\in \R^n: |\xi_2-s|>\frac{s}{2n}\}\cap N_{s}(\Sigma(s))^0,\\
V_j(s)&=\{\xi\in \R^n: |\xi_j|>\frac{s}{2n}\}\cap N_{s}(\Sigma(s))^0, \quad j=3,\dots, n,
\end{aligned}
\end{equation}
where $N_{s}(\Sigma(s))^0$ is the interior of $N_{s}(\Sigma(s))$, 
we have 
\begin{equation}
\label{eq_5_open_cover_ind_s_1}
N_{\frac{s}{2n}}(\Sigma(s))\subset\cup_{j=2}^n V_j(s). 
\end{equation}

It will be convenient to decompose the open sets $V_j(s)$ further,
\begin{equation}
\label{eq_5_sets_V_j_pm}
\begin{aligned}
V_{2,\pm}(s)&=\{\xi\in V_2(s): \pm(\xi_2-s)>0 \},\\
 V_{j,\pm}(s)&=\{\xi\in V_j(s): \pm \xi_j>0 \},  \quad j=3,\dots, n.
\end{aligned}
\end{equation}

We have the scaling relations, 
\begin{equation}
\label{eq_scalling_relations_sets}
\begin{aligned}
&\Sigma(s)=s\Sigma(1), \quad N_{\frac{s}{2n}}(\Sigma(s))=s N_{\frac{1}{2n}}(\Sigma(1)),\\
&V_1(s)=s V_1(1), \quad V_{j,\pm}(s)=s V_{j,\pm}(1), \quad j=2, 3,\dots, n.
\end{aligned}
\end{equation}
Thus, \eqref{eq_5_open_cover_ind_s_1} is equivalent to 
\begin{equation}
\label{eq_5_open_cover_ind_s}
N_{\frac{1}{2n}}(\Sigma(1))\subset \cup_{j=2}^{n} (V_{j,+}(1)\cup V_{j,-}(1)).
\end{equation}

Let $\chi_{j,\pm}(1)$, $j=2,\dots, n$, be a partition of unity subordinate to the open cover \eqref{eq_5_open_cover_ind_s} of the compact set $N_{\frac{1}{2n}}(\Sigma(1))$, i.e.
$0\le \chi_{j,\pm}(1)\in C^\infty_0(V_{j,\pm}(1))$,  and $\sum_{j=2}^n (\chi_{j,+}(1)+\chi_{j,-}(1))=1$ near $N_{\frac{1}{2n}}(\Sigma(1))$. We set 
$\chi_1(1)=1-\sum_{j=2}^n (\chi_{j,+}(1)+\chi_{j,-}(1))\in C^\infty(\R^n)$ and we have $\chi_1(1)=0$ near $N_{\frac{1}{2n}}(\Sigma(1))$.  

Setting 
\begin{equation}
\label{eq_5_partician_scalling}
\chi_1(s)(\xi)=\chi_1(1)(\xi/s),\quad \chi_{j,\pm}(s)(\xi)=\chi_{j,\pm}(1)(\xi/s),\quad j=2,\dots,n,
\end{equation}
we see that $\chi_{j,\pm}(s)$, $j=2,\dots, n$, is a partition of unity subordinate to the open cover $\cup_{j=2}^{n} (V_{j,+}(s)\cup V_{j,-}(s))$ of the compact set $N_{\frac{s}{2n}}(\Sigma(s))$, i.e.
$0\le \chi_{j,\pm}(s)\in C^\infty_0(V_{j,\pm}(s))$,  and $\sum_{j=2}^n (\chi_{j,+}(s)+\chi_{j,-}(s))=1$ near $N_{\frac{s}{2n}}(\Sigma(s))$. Furthermore, $\chi_1(s)=1-\sum_{j=2}^n (\chi_{j,+}(s)+\chi_{j,-}(s))\in C^\infty(\R^n)$ and $\chi_1(s)=0$ near $N_{\frac{s}{2n}}(\Sigma(s))$.

To solve the equation 
\begin{equation}
\label{eq_5_0_1}
(|\xi|^2-2i \zeta\cdot \xi)^m \hat w=1\quad \text{in}\quad \R^n,
\end{equation}
 we seek a tempered distribution $\hat w$ of the form
\begin{equation}
\label{eq_5_0_2}
\hat w=\hat w_1+\sum_{j=2}^n (\hat w_{j,+}+\hat w_{j,-}),
\end{equation}
where 
$\hat w_1$ satisfies the equation
\begin{equation}
\label{eq_5_1_hat_w_1}
(|\xi|^2-2i \zeta\cdot \xi)^m \hat w_1=\chi_1(s)\quad \text{in}\quad \R^n,
\end{equation}
and 
$\hat w_{j,\pm}$ solves the equation
\begin{equation}
\label{eq_5_1}
(|\xi|^2-2i \zeta\cdot \xi)^m \hat w_{j,\pm}= \chi_{j,\pm}(s)\quad \text{in}\quad \R^n,\quad j=2,\dots,n.
\end{equation}

Since $\supp (\chi_1(s))\subset \R^n\setminus N_{\frac{s}{2n}}(\Sigma(s))$, we see that the function 
\begin{equation}
\label{eq_5_0_3}
\hat w_1(\xi)=\frac{\chi_1(s)(\xi)}{(p_\zeta(\xi))^m}\in C^\infty(\R^n)\cap \mathcal{S}'(\R^n)
\end{equation}
solves the equation \eqref{eq_5_1_hat_w_1}.

To solve the equation \eqref{eq_5_1}, we define $\Phi^{(j,\pm)}(s): V_{j,\pm}(s)\to \Phi^{(j,\pm)}(s)(V_{j,\pm}(s))$, $j=2,\dots,n$,  by
\begin{equation}
\label{eq_5_diffeomorphism_def}
\begin{aligned}
\eta_1&=\Phi_1^{(j,\pm)}(s)(\xi)=-2\xi_1,\\
\eta_j&=\Phi_j^{(j,\pm)}(s)(\xi) =\frac{\xi_1^2+(\xi_2-s)^2+\xi_3^2+\dots+\xi_n^2-s^2}{s},\\
\eta_l& =\Phi_l^{(j,\pm)}(s)(\xi)=\xi_l,\quad l\ne 1,\ j.
\end{aligned}
\end{equation}
The Jacobian of this transformation is given by
\[
|\det [\frac{\p \eta}{\p \xi}]|=\begin{cases}  \frac{4|\xi_2-s|}{s}, & j=2,\\
\frac{4|\xi_j|}{s}, & j=3,\dots, n. 
\end{cases}
\]
Thus, for $\xi\in V_{j,\pm}(s)$, we have $\frac{2}{n}<|\det [\frac{\p \eta}{\p \xi}]|<8$ and hence,  $\Phi^{(j,\pm)}(s)$ is a local diffeomorphism. Furthermore, since  $\Phi^{(j,\pm)}(s): V_{j,\pm}(s)\to \Phi^{(j,\pm)}(s)(V_{j,\pm}(s))$ is injective, we conclude that it is a global diffeomorphism.

We have also the  scaling relation, 
\[
\Phi^{(j,\pm)}(s)(\xi)=s\Phi^{(j,\pm)}(1)(\xi/s),
\]
where the map
\[
\Phi^{(j,\pm)}(1): V_{j,\pm}(1)\to \Phi^{(j,\pm)}(1) (V_{j,\pm}(1))
\]
is a smooth diffeomorphism between bounded open sets, which are independent of $s$.

Now in the new coordinates  we have
\[
 p_\zeta((\Phi^{(j,\pm)}(s))^{-1}(\eta))=s(\eta_j+i\eta_1),
 \]
 and therefore, the equation \eqref{eq_5_1} becomes
 \begin{equation}
 \label{eq_5_1_2}
 s^m(\eta_j+i\eta_1)^m \hat w_{j,\pm}((\Phi^{(j,\pm)}(s))^{-1}(\eta))=\chi_{j,\pm}(s)((\Phi^{(j,\pm)}(s))^{-1}(\eta))\quad  \text{in}\quad \R^n.
 \end{equation}

To proceed we shall need the following result.  
\begin{lem}
\label{lem_fundam_solution}
Let $m\ge 2$, $j=2,\dots, n$, and let  
\begin{equation}
\label{eq_e_zeta_m_j}
E_\zeta^{(m, j)}=\frac{(-1)^{m-1}}{s^m(m-1)!}\p_{\eta_j}^{(m-1)}\bigg( \frac{1}{\eta_j+i\eta_1} \bigg)\in \mathcal{S}'(\R^n),
\end{equation}
where the derivatives are taken in the sense of distributions. 
Then 
\begin{equation}
\label{eq_5_2}
s^m(\eta_j+i\eta_1)^m E_\zeta^{(m,j)}=1\quad \text{in}\quad \R^n.
\end{equation}
\end{lem}

\begin{proof}
To fix the ideas let us take $j=2$. Let $\varphi\in \mathcal{S}(\R^n)$.  We have
\begin{equation}
\label{eq_5_3}
\begin{aligned}
\langle s^m(\eta_2&+i\eta_1)^m E_\zeta^{(m,2)},\varphi \rangle=\frac{1}{(m-1)!}\int_{\R^n} \frac{1}{\eta_2+i\eta_1} \p_{\eta_2}^{(m-1)}( (\eta_2+i\eta_1)^m \varphi(\eta))d\eta\\
&= \frac{1}{(m-1)!}\lim_{\varepsilon\to 0}\int_{\R^{n-2}}\int_{\R^2\setminus  B_\varepsilon(0)} \frac{1}{\eta_2+i\eta_1} \p_{\eta_2}^{(m-1)}( (\eta_2+i\eta_1)^m \varphi(\eta))d\eta_1 d\eta_2 d\eta'',
\end{aligned}
\end{equation}
where $\eta=(\eta_1,\eta_2,\eta'')$ and $B_\varepsilon(0)=\{(\eta_1,\eta_2)\in \R^2: |\eta_1|^2+|\eta_2|^2\le \varepsilon^2\}$.
Here and in what follows $\langle \cdot, \cdot \rangle$ denotes the distributional duality.

Integrating by parts and recalling that $(\eta_2+i\eta_1)^m\varphi\in \mathcal{S}(\R^n)$, we get
\begin{align*}
I_\varepsilon:&=\int_{\R^2\setminus  B_\varepsilon(0)} \frac{1}{\eta_2+i\eta_1} \p_{\eta_2}^{(m-1)}( (\eta_2+i\eta_1)^m \varphi(\eta))d\eta_1 d\eta_2\\
&= (-1)^{m-1}  \int_{\R^2\setminus B_\varepsilon(0)} \p_{\eta_2}^{(m-1)}\bigg(  \frac{1}{\eta_2+i\eta_1}\bigg) (\eta_2+i\eta_1)^m \varphi(\eta)d\eta_1 d\eta_2\\
&+\sum_{k=0}^{m-2} (-1)^{m-2-k}\int_{\p B_\varepsilon(0)} \nu_2   \p_{\eta_2}^{(m-2-k)}\bigg(  \frac{1}{\eta_2+i\eta_1}\bigg)  \p_{\eta_2}^{(k)}( (\eta_2+i\eta_1)^m \varphi(\eta)) dS,
\end{align*}
where $dS$ is the Euclidean arc measure on $\p B_\varepsilon(0)$ and $\nu=(\nu_1,\nu_2)$ is the interior unit normal to $\p B_\varepsilon(0)$.  

Writing $\eta_1=\varepsilon\cos\theta$, $\eta_2=\varepsilon\sin\theta$  on $\p B_\varepsilon(0)$, and using the facts that
\begin{equation}
\label{eq_5_4}
\p_{\eta_2}^{(l)} \bigg(\frac{1}{\eta_2+i\eta_1}\bigg)=(-1)^l l! \frac{1}{(\eta_2+i\eta_1)^{l+1}}
\end{equation}
and 
\[
\p_{\eta_2}^{(k)}( (\eta_2+i\eta_1)^m \varphi(\eta))=\mathcal{O}(|(\eta_1,\eta_2)|^{m-k}), \quad k\le m,
\]
we see that 
\[
\bigg| \int_{\p B_\varepsilon(0)} \nu_2   \p_{\eta_2}^{(m-2-k)}\bigg(  \frac{1}{\eta_2+i\eta_1}\bigg)  \p_{\eta_2}^{(k)}( (\eta_2+i\eta_1)^m \varphi(\eta)) dS \bigg|\le \int_0^{2\pi}\mathcal{O}(\varepsilon^2)d\theta \to 0, 
\]
as $\varepsilon \to 0$. Therefore, also using \eqref{eq_5_4}, we obtain that 
\[
\lim_{\varepsilon\to 0} I_\varepsilon= (m-1)!\int_{\R^2} \varphi(\eta)d\eta_1 d\eta_2.
\]
This together with \eqref{eq_5_3} implies \eqref{eq_5_2}. The proof is complete. 
\end{proof}
 
Lemma \ref{lem_fundam_solution} implies that the distribution
\[
\chi_{j,\pm}(s)((\Phi^{(j,\pm)}(s))^{-1}(\eta)) E_\zeta^{(m,j)}\in \mathcal{E}'(\R^n)
\] 
is a solution of \eqref{eq_5_1_2}.  Thus, 
\begin{equation}
\label{eq_5_5}
\hat w_{j,\pm}(\xi)=(\Phi^{(j,\pm)}(s))^*( (\chi_{j,\pm}(s) \circ (\Phi^{(j,\pm)}(s))^{-1})  E_\zeta^{(m,j)})
\in \mathcal{E}'(V_{j,\pm}(s))
\end{equation}
is a solution of  \eqref{eq_5_1}. Here $(\Phi^{(j,\pm)}(s))^*$ is the pullback  by the diffeomorphism $\Phi^{(j,\pm)}(s)$, see \cite{Friedlander_book}.

Let
\begin{equation}
\label{eq_5_fundam_sol_f}
\begin{aligned}
 \mathcal{S}'(\R^n)\ni E_\zeta^{(m)}=\hat w=\frac{\chi_1(s)}{(p_\zeta)^m} &+
\sum_{j=2}^n  (\Phi^{(j,+)}(s))^*( (\chi_{j,+}(s) \circ (\Phi^{(j,+)}(s))^{-1})  E_\zeta^{(m,j)})\\ 
&+ \sum_{j=2}^n  (\Phi^{(j,-)}(s))^*( (\chi_{j,-}(s) \circ (\Phi^{(j,-)}(s))^{-1})  E_\zeta^{(m,j)}).
\end{aligned}
\end{equation}

Summing up the discussion so far, we have the following result. 
\begin{prop}
\label{prop_fundam_solution_polyharmonic}
The  distribution $g_\zeta^{(m)}= \mathcal{F}^{-1} (E_\zeta^{(m)})$  is a tempered fundamental solution of the operator $(-\Delta-2\zeta\cdot \nabla)^m$.  The convolution operator 
\begin{equation}
\label{eq_solution_op}
G_\zeta^{(m)} f=g_\zeta^{(m)}* f
\end{equation}
is  continuous $\mathcal{S}(\R^n)\to \mathcal{S}'(\R^n)$, and $w=G_\zeta^{(m)}f$ is a solution to the equation
\[
(-\Delta-2\zeta\cdot\nabla)^m w=f\quad \text{in}\quad \R^n.
\]
\end{prop}

Our next goal is to extend the convolution operator $G_\zeta^{(m)}$ to a bounded operator between suitable  weighed $L^2$ spaces, introduced in \eqref{eq_int_weighted_L_2},  and to  obtain  estimates for the corresponding operator norm.  To that end, we shall prepare by proving some auxiliary results.  

\begin{lem}
\label{lem_diffeom_invarience}
Let $W_1$ and $W_2$ be open subsets of $\R^n$,  let  $\Phi:W_1\to W_2$ be a $C^\infty$ diffeomorphism and let $W_2'\Subset W_2$ be open. Then for $\sigma \in \R$, we have 
\begin{equation}
\label{eq_5_6}
\|\mathcal{F}^{-1} (\Phi^* u)\|_{L^2_\sigma(\R^n)}\le C\|\mathcal{F}^{-1}(u)\|_{L^2_\sigma(\R^n)}, \quad u\in \mathcal{E}'(W_2'),\quad \mathcal{F}^{-1}(u)\in L^2_\sigma(\R^n).
\end{equation}
Here the constant $C$ depends only on $\sigma$,  $\|\p^{\alpha} \Phi\|_{L^\infty(\Phi^{-1}(W_2'))}$, $\|\p^{\alpha} \Phi^{-1}\|_{L^\infty(W_2')}$ for $|\alpha|\ge 1$.
\end{lem}

\begin{proof}
First notice that 
\[
\|\mathcal{F}^{-1}(u)\|_{L^2_\sigma(\R^n)}^2=(2\pi)^{-n}\int_{\R^n}(1+|\xi|^2)^{\sigma} |\mathcal{F}(u)(-\xi)|^2d\xi =  (2\pi)^{-n}\|u\|_{H^\sigma(\R^n)}^2. 
\]
Therefore, \eqref{eq_5_6} is equivalent to the fact that
\[
\|\Phi^* u\|_{H^\sigma(\R^n)} \le C\|u\|_{H^\sigma(\R^n)}, \quad u\in \mathcal{E}'(W_2')\cap H^\sigma(\R^n),
\]
which expresses the invariance of the Sobolev space $H^\sigma(\R^n)$ under a smooth diffeomorphism, see   \cite[Theorem 2.6.1]{Hormander_book_linear}. The proof is complete. 
\end{proof}

We shall have to apply Lemma \ref{lem_diffeom_invarience} to the diffeomorphisms $\Phi^{(j,\pm)}(s)$, $j=2,\dots, n$, introduced in \eqref{eq_5_diffeomorphism_def},  which depend on a large parameter $s$, and to make sure that the constants in \eqref{eq_5_6} are independent of $s$ we require the following result. 

\begin{lem}
\label{lem_diffeom_invarience_s}
Let $\Phi^{(j,\pm)}(s)$, $j=2,\dots, n$,  be the diffeomorphism, defined in \eqref{eq_5_diffeomorphism_def}. Then for all $s\ge s_0>0$, we have
\[
\| \p^\alpha \Phi^{(j,\pm)}(s) \|_{L^\infty(V_{j,\pm}(s))}\le C_\alpha,\ \|\p^{\alpha}(\Phi^{(j,\pm)}(s))^{-1}\|_{L^\infty
(\Phi^{(j,\pm)}(s)(V_{j,\pm}(s)))
}\le C_\alpha, \ |\alpha|\ge 1,
\]
uniformly in $s$.  
\end{lem}

\begin{proof}
Recall that 
\[
V_{j,\pm}(s)=sV_{j,\pm}(1),\quad
\Phi^{(j,\pm)}(s)(\xi)=s\Phi^{(j,\pm)}(1)(\xi/s),
\]
and therefore, 
\begin{equation}
\label{eq_lem_diff_s_1}
\p^\alpha_\xi (\Phi^{(j,\pm)}(s))(\xi) =s^{1-|\alpha|} \p^\alpha_\xi (\Phi^{(j,\pm)}(1))(\xi/s).
\end{equation}

Now writing 
\[
\eta=\Phi^{(j,\pm)}(s)(\xi)=s\Phi^{(j,\pm)}(1)(\xi/s),
\]
we see that 
\[
(\Phi^{(j,\pm)}(s))^{-1}(\eta)=\xi=s (\Phi^{(j,\pm)}(1))^{-1}(\eta/s).
\]
Hence, 
\begin{equation}
\label{eq_lem_diff_s_2}
\p^{\alpha}_\eta((\Phi^{(j,\pm)}(s))^{-1})(\eta)=s^{1-|\alpha|} \p_\eta^{\alpha}(\Phi^{(j,\pm)}(1))^{-1}(\eta/s).
\end{equation}

The map
\[
\Phi^{(j,\pm)}(1): V_{j,\pm}(1)\to \Phi^{(j,\pm)}(1)(V_{j,\pm}(1))
\]
is a smooth diffeomorphism between bounded open sets, which are independent of $s$, and it can easily be seen to extend to a smooth diffeomorphism on  a neighborhood of $\overline{V_{j,\pm}(1)}$. The claim of the Lemma then follows from \eqref{eq_lem_diff_s_1} and \eqref{eq_lem_diff_s_2}. The proof is complete. 
\end{proof}

We shall need the following consequence of \cite[Lemma 3.1]{Sylvester_Uhlmann_1987}. 
\begin{lem}
\label{lem_SU_modified}
Let $n\ge 3$, $x=(x_1,x_2,x'')\in \R^n$,  and let 
\[
H(x_1,x_2)=\frac{1}{|x_2+ix_1|}\in L^1_{\emph{\text{loc}}}(\R^2)\cap\mathcal{S}'(\R^2).
\] 
Then for  $-1<\sigma<0$, 
\begin{equation}
\label{eq_SU_module}
\| (H(x_1,x_2)\otimes \delta(x''))* f\|_{L^2_\sigma}\le C\|f\|_{L^2_{\sigma+1}},\quad f\in\mathcal{S}(\R^n).
\end{equation}
\end{lem}

\begin{proof}
Writing $x'=(x_1,x_2)$,  we have
\[
(H\otimes \delta(x''))* f=(H\otimes \delta(x''))_y(f(x-y))=\int_{\R^2} \frac{f(x'-y',x'')}{|y_2+i y_1|}dy_1dy_2.
\]
By inspection of the proof of \cite[Lemma 3.1]{Sylvester_Uhlmann_1987}, we get 
\begin{align*}
\int_{\R^2} (1+|x'|^2)^\sigma& \bigg| \int_{\R^2} \frac{f(x'-y',x'')}{|y_2+i y_1|}dy_1dy_2 \bigg|^2 dx_1 dx_2\\
&\le C
\int_{\R^2} (1+|x'|^2)^{\sigma+1} |f(x',x'')|^2dx_1dx_2,
\end{align*}
for all $x''\in \R^{n-2}$. Using the fact that $\sigma<0$
and $\sigma+1>0$, we obtain that 
\begin{align*}
\int_{\R^2} (1+|x|^2)^\sigma& \bigg| \int_{\R^2} \frac{f(x'-y',x'')}{|y_2+i y_1|}dy_1dy_2 \bigg|^2 dx_1 dx_2\\
&\le C
\int_{\R^2} (1+|x|^2)^{\sigma+1} |f(x',x'')|^2dx_1dx_2,
\end{align*}
and hence, the estimate \eqref{eq_SU_module} follows by integration with respect to $x''$. The proof is complete. 
\end{proof}

Lemma \ref{lem_SU_modified} will be used in the proof of the following result which will be needed later.

\begin{lem}
\label{lem_SU_modified_second}
Let $n\ge 3$,   $x=(x_1,x_2,x'')\in \R^n$, and $m=1,2,\dots$. Then for $-m<\sigma<1-m$, we have 
\begin{equation}
\label{eq_SU_module_next}
\bigg\| \bigg(\frac{x_2^{m-1}}{x_2+ix_1}\otimes \delta(x'')\bigg)*g  \bigg\|_{L^2_\sigma}\le C \|g\|_{L^2_{\sigma+2m-1}}, \quad g\in \mathcal{S}(\R^n).
\end{equation}
\end{lem}

\begin{proof}
Writing $x'=(x_1,x_2)$, we have 
\begin{align*}
\bigg|  \bigg(\frac{x_2^{m-1}}{x_2+ix_1}\otimes \delta(x'')\bigg)*g \bigg|&\le \int_{\R^2}\frac{|y_2|^{m-1}}{|y_2+i y_1|} |g(x'-y',x'')|dy''\\
&\le C|x'|^{m-1} \int_{\R^2}\frac{1}{|y_2+i y_1|} |g(x'-y',x'')|dy''\\
&+ C \int_{\R^2}\frac{|x'-y'|^{m-1}}{|y_2+i y_1|} |g(x'-y',x'')|dy''\\
&\le C(1+|x|^2)^{(m-1)/2} \bigg( \bigg( \frac{1}{|x_2+ix_1|}\otimes \delta(x'')\bigg)*|g| \\
&+ 
 \bigg( \frac{1}{|x_2+ix_1|}\otimes \delta(x'')\bigg)*(|\cdot|^{m-1}|g|)
 \bigg).
\end{align*}
Thus, 
\begin{align*}
\bigg\| \bigg(\frac{x_2^{m-1}}{x_2+ix_1}\otimes \delta(x'')\bigg)*g  \bigg\|_{L^2_\sigma}&\le 
C \bigg\| \bigg( \frac{1}{|x_2+ix_1|}\otimes \delta(x'')\bigg)*|g| \bigg\|_{L^2_{\sigma+m-1}}\\
&+C \bigg\| \bigg( \frac{1}{|x_2+ix_1|}\otimes \delta(x'')\bigg)*(|\cdot|^{m-1}|g|) \bigg\|_{L^2_{\sigma+m-1}}.
\end{align*}
As $-1<\sigma+m-1<0$, applying Lemma \ref{lem_SU_modified}, we see that the expression above does not exceed  
\[
 C\|g\|_{L^2_{\sigma+m}}+C\||\cdot|^{m-1} g\|_{L^2_{\sigma+m}}\le C\|g\|_{L^2_{\sigma+2m-1}},
\]
which shows \eqref{eq_SU_module_next}. The proof is complete. 
\end{proof}

For future reference we shall also need the following result.  
\begin{lem}
\label{lem_eq_5_9}
We have
\begin{equation}
\label{eq_lem_eq_5_9}
\| f*\varphi  \|_{L^2_{\sigma}}\le C_\varphi \| f \|_{L^2_{\sigma}}, \quad f\in L^2_\sigma(\R^n),\quad \varphi\in \mathcal{S}(\R^n),
\end{equation}
where $\sigma>0$.  
\end{lem}
\begin{proof}
Using that
\[
(1+|x|^2)^{\sigma/2}\le C((1+|x-y|^2)^{\sigma/2} +(1+|y|^2)^{\sigma/2}), 
\] 
we have
\begin{align*}
\| f*\varphi  \|_{L^2_{\sigma}}&\le \bigg( \int \bigg(\int (1+|x|^2)^{\sigma/2} |\varphi(x-y)| |f(y)| dy \bigg)^2 dx  \bigg)^{1/2}\\
&\le C(\| (1+|\cdot|^2)^{\sigma/2}|\varphi|*|f|\|_{L^2} +\|  |\varphi|* (1+|\cdot|^2)^{\sigma/2} |f| \|_{L^2})\\
&\le C (\| (1+|\cdot|^2)^{\sigma/2}\varphi  \|_{L^1}\|f\|_{L^2}+ \|\varphi\|_{L^1}\|f\|_{L^2_\sigma})\le C_\varphi\|f\|_{L^2_\sigma},
\end{align*}  
where in the last line we have used  Young's inequality for convolutions. The proof is complete. 
\end{proof}

In our considerations we shall apply Lemma \ref{lem_eq_5_9} to the function $\varphi=\mathcal{F}^{-1}(\chi_{j,\pm}(s))$, $j=2,\dots,n$, where $\chi_{j,\pm}(s)$ is defined by \eqref{eq_5_partician_scalling}, and consequently depends on the large parameter $s$. In order to conclude that the constant in \eqref{eq_lem_eq_5_9} is independent of $s$, we shall prove the following result.  Notice that here it is important that our partition of unity is chosen so that it respects the scaling relations 
\eqref{eq_scalling_relations_sets}.

\begin{lem}
\label{lem_eq_5_9_chi}
Let  $\sigma>0$ and let $\chi_{j,\pm}(s)$ be defined by \eqref{eq_5_partician_scalling}. Then the norms
\[
\|\mathcal{F}^{-1}(\chi_{j,\pm}(s))\|_{L^1(\R^n)},\quad \| (1+|\cdot|^2)^{\sigma/2}\mathcal{F}^{-1}(\chi_{j,\pm} (s)) \|_{L^1(\R^n)}, \quad j=2,\dots, n, 
\]
are $\mathcal{O}(1)$,  uniformly in $s\ge s_0>0$. 
\end{lem}

\begin{proof}
By  \eqref{eq_5_partician_scalling}, we get 
\[
\mathcal{F}^{-1}(\chi_{j,\pm}(s))(x)=s^n \mathcal{F}^{-1}(\chi_{j,\pm}(1))(sx),
\]
and therefore, 
\begin{equation}
\label{lem_eq_5_9_chi_1}
\|\mathcal{F}^{-1}(\chi_{j,\pm}(s))\|_{L^1(\R^n)}=\|\mathcal{F}^{-1}(\chi_{j,\pm}(1))\|_{L^1(\R^n)}.
\end{equation}
For $s\ge s_0>0$, as $\sigma>0$,  we also have 
\begin{equation}
\label{lem_eq_5_9_chi_2}
\| (1+|\cdot|^2)^{\sigma/2}\mathcal{F}^{-1}(\chi_{j,\pm} (s)) \|_{L^1(\R^n)}\le \int_{\R^n} \bigg(1+\frac{|y|^2}{s_0^2}\bigg)^{\sigma/2}|\mathcal{F}^{-1}(\chi_{j,\pm}(1))(y)|dy.
\end{equation}
As $\mathcal{F}^{-1}(\chi_{j,\pm}(1))\in \mathcal{S}(\R^n)$, the expressions in \eqref{lem_eq_5_9_chi_1} and \eqref{lem_eq_5_9_chi_2} are finite. 
The proof is complete. 
\end{proof}

We are now ready to prove the following result. 
\begin{prop}
\label{prop_G_s_m_L_2}
Let $m=1,2,\dots$, and $-m<\sigma < 1-m$. Then for $|\zeta|\ge s_0>0$,  the  operator $G_\zeta^{(m)}$ can be extended to a bounded operator $L^2_{\sigma+2m-1}(\R^n)\to L^2_\sigma(\R^n)$ such that
\begin{equation}
\label{eq_5_10}
\| G_\zeta^{(m)}f\|_{L^2_\sigma}\le \frac{C}{|\zeta|^m} \|f\|_{L^2_{\sigma+2m-1}}.
\end{equation}
\end{prop}

\begin{proof}
It suffices to prove \eqref{eq_5_10} when $f\in \mathcal{S}(\R^n)$. When doing so we shall make use of the fact that 
\[
G_\zeta^{(m)} f=\mathcal{F}^{-1}(E_\zeta^{(m)})*f,
\]
where $E_\zeta^{(m)}$ is given by
\eqref{eq_5_fundam_sol_f}. 

First notice that 
\[
\supp(\chi_1(s))\subset\{\xi\in \R^n: \dist(\xi,\Sigma_\zeta)>\frac{s}{2n}\},
\]
and therefore, an application of Lemma \ref{lem_behavior_p_zeta} shows that 
\[
|p_\zeta(\xi)|\gtrsim s^2,\quad \xi\in \supp(\chi_1(s)). 
\]
By Parseval's formula and using that $\sigma<0$ and $\sigma+2m-1>0$, we have
\begin{equation}
\label{eq_5_11}
\begin{aligned}
\bigg\| \mathcal{F}^{-1}\bigg(\frac{\chi_1(s)}{(p_\zeta(\xi))^m}\bigg)*f \bigg\| _{L^2_\sigma}& \le \bigg\| \mathcal{F}^{-1}\bigg(\frac{\chi_1(s)}{(p_\zeta(\xi))^m}\bigg)*f \bigg\| _{L^2}=(2\pi)^{-n/2} \bigg\| \frac{\chi_1(s)}{(p_\zeta(\xi))^m}\hat f \bigg\| _{L^2}\\
&\le\frac{C}{|\zeta|^{2m}}\|f\|_{L^2} \le\frac{C}{|\zeta|^{2m}}\|f\|_{L^2_{\sigma+2m-1}}.
\end{aligned}
\end{equation}

Let now $1<j\le  n$, and assume to fix the ideas that $j=2$, all other cases being identical.  Writing $x=(x_1,x_2,x'')$, by \eqref{eq_e_zeta_m_j}, we have
\begin{equation}
\label{eq_5_8}
\begin{aligned}
\mathcal{F}^{-1}(  E_\zeta^{(m,2)})&=\frac{(-1)^{m-1}}{s^m(m-1)!}\mathcal{F}^{-1}\bigg(\p_{\eta_2}^{(m-1)} \bigg( \frac{1}{\eta_2+i\eta_1}\bigg)\bigg)\\
&=\frac{i^{m-1}}{s^m(m-1)!} x_2^{m-1} \mathcal{F}^{-1}\bigg( \frac{1}{\eta_2+i\eta_1}\bigg)
\\
&=\frac{i^{m} x_2^{m-1}}{2\pi s^m(m-1)! (x_2+i x_1)}\otimes \delta(x''),
\end{aligned}
\end{equation}
see \cite[Exercise 7.1.40]{Hormander_book_1}.

Using that 
\[
\Phi^*(fu)=(\Phi^* f) (\Phi^*u), \quad f\in C^\infty(\R^n),\quad u\in \mathcal{D}'(\R^n),
\]
see \cite[p. 135]{Hormander_book_1},
for $s\ge s_0>0$, with the help of Lemma \ref{lem_diffeom_invarience} combined with Lemma \ref{lem_diffeom_invarience_s},   we get  
\begin{equation}
\label{eq_5_12}
\begin{aligned}
\|\mathcal{F}^{-1}& ((\Phi^{(2,\pm)}(s))^*( (\chi_{2,\pm}(s) \circ  (\Phi^{(2,\pm)}(s))^{-1})  E_\zeta^{(m,2)})) *f\|_{L^2_\sigma}\\
&= 
\| \mathcal{F}^{-1}\big(  (\Phi^{(2,\pm)}(s))^* \big( (\hat f\chi_{2,\pm}(s))\circ (\Phi^{(2,\pm)}(s))^{-1}   E_\zeta^{(m,2)}  \big) \big)\|_{L^2_\sigma}\\
& \le C\|  \mathcal{F}^{-1}\big( (\hat f\chi_{2,\pm}(s))\circ (\Phi^{(2,\pm)}(s))^{-1}   E_\zeta^{(m,2)}  \big)  \|_{L^2_\sigma}\\
&=  C\| \mathcal{F}^{-1}(  E_\zeta^{(m,2)})* \mathcal{F}^{-1}\big( (\hat f\chi_{2,\pm}(s))\circ (\Phi^{(2,\pm)}(s))^{-1} \big)  \|_{L^2_\sigma}\\
& \le \frac{C}{s^m}\bigg\| \bigg(\frac{ x_2^{m-1}}{x_2+i x_1} \otimes \delta(x'')\bigg) * \mathcal{F}^{-1}\big( (\hat f\chi_{2,\pm}(s))\circ (\Phi^{(2,\pm)}(s))^{-1} \big)   \bigg\|_{L^2_\sigma}.
\end{aligned}
\end{equation}
In the last line we have used \eqref{eq_5_8}.  

An application of Lemma \ref{lem_SU_modified_second} shows that the last expression can be estimated as follows, 
\begin{equation}
\label{eq_5_12_second_part}
\begin{aligned}
& \le \frac{C}{s^m} \|  \mathcal{F}^{-1}\big( (\hat f\chi_{2,\pm}(s))\circ (\Phi^{(2,\pm)}(s))^{-1} \big)  \|_{L^2_{\sigma+2m-1}}
\le \frac{C}{s^m} \| \mathcal{F}^{-1} (\hat f\chi_{2,\pm}(s))  \|_{L^2_{\sigma+2m-1}}\\
&= \frac{C}{s^m} \| f*\mathcal{F}^{-1} (\chi_{2,\pm}(s))  \|_{L^2_{\sigma+2m-1}} \le \frac{C}{s^m} \| f \|_{L^2_{\sigma+2m-1}}.
\end{aligned}
\end{equation}
Here we have used Lemma \ref{lem_diffeom_invarience}, combined with Lemma \ref{lem_diffeom_invarience_s}, and  Lemma 
\ref{lem_eq_5_9}, combined with Lemma \ref{lem_eq_5_9_chi}.

The estimate \eqref{eq_5_10} follows from \eqref{eq_5_11}, \eqref{eq_5_12} and \eqref{eq_5_12_second_part}. The proof is complete. 
\end{proof}

We now proceed to discuss estimates for the convolution operator $G_\zeta^{(m)}$ in suitable $L^p$ spaces. 
\begin{prop}
\label{prop_uniform_l_p_poly}
Let $\zeta\in \C^n\setminus\{0\}$, $\zeta\cdot\zeta=0$. Then 
$G_\zeta^{(m)}$ can be extended to a bounded operator $L^{\frac{2n}{n+2m}}(\R^n)\to L^{\frac{2n}{n-2m}}(\R^n)$  such that 
\begin{equation}
\label{eq_5_13_0}
\| G_\zeta^{(m)} f\|_{L^{\frac{2n}{n-2m}}}\le C\| f\|_{L^{\frac{2n}{n+2m}}}, \quad f\in L^{\frac{2n}{n+2m}}(\R^n),
\end{equation}
uniformly in $\zeta$. 
\end{prop}

\begin{proof}
First 
as a consequence of \eqref{eq_KRS}, we have the following estimate
\begin{equation}
\label{eq_5_13}
\|u\|_{L^{\frac{2n}{n-2m}}}\le C\|(-\Delta-2\zeta\cdot\nabla)^m u\|_{L^{\frac{2n}{n+2m}}}
\end{equation}
uniformly in $\zeta\in \C^n$, $\zeta\cdot\zeta=0$,  and $u\in \mathcal{S}(\R^n)$. 

Next we would like to substitute $u=G_\zeta^{(m)}f$, $f\in \mathcal{S}(\R^n)$, in \eqref{eq_5_13}. However, the operator $G_\zeta^{(m)}$ does not preserve the Schwartz space. To overcome this difficulty, let us consider the space of $\mathcal{S}(\R^n)$, given by
\[
X_\zeta=\{f\in \mathcal{S}(\R^n): \hat f\in C^\infty_0(\R^n\setminus \Sigma_\zeta)\}.
\]
Let us show that 
\begin{equation}
\label{eq_5_14}
G_\zeta^{(m)}: X_\zeta\to X_\zeta.
\end{equation}
Indeed, let $f\in X_\zeta$. Using \eqref{eq_5_fundam_sol_f}, we get 
\begin{align*}
\widehat{G_\zeta^{(m)}f}=\frac{\hat f\chi_1(s)}{(p_\zeta)^m}&+\sum_{j=2}^n  (\Phi^{(j,+)}(s))^* \big( (\hat f \chi_{j,+}(s)\circ (\Phi^{(j,+)}(s))^{-1}) E_\zeta^{(m,j)}\big)\\
 &+  \sum_{j=2}^n (\Phi^{(j,-)}(s))^* \big( (\hat f \chi_{j,-}(s)\circ (\Phi^{(j,-)}(s))^{-1}) E_\zeta^{(m,j)}\big).
\end{align*}
Notice that $\hat f\chi_{j,\pm}(s)\circ (\Phi^{(j,\pm)}(s))^{-1}\in C_0^\infty(\R^n\setminus\{\eta_j=0,\eta_1=0\})$. Therefore, for $\varphi\in \mathcal{S}(\R^n)$, we have 
\begin{align*}
\langle  (\hat f \chi_{j,\pm}(s)&\circ (\Phi^{(j,\pm)}(s))^{-1}) E_\zeta^{(m,j)}, \varphi\rangle\\
&=\frac{1}{s^m(m-1)!}\int_{\R^n} \frac{1}{\eta_j+i\eta_1}
\p_{\eta_j}^{(m-1)} ( \hat f \chi_{j,\pm}(s)\circ (\Phi^{(j,\pm)}(s))^{-1}\varphi)d\eta\\
&=\int_{\R^n}  \frac{1}{s^m(\eta_j+i\eta_1)^m} \hat f \chi_{j,\pm}(s)\circ (\Phi^{(j,\pm)}(s))^{-1}\varphi d\eta.
\end{align*}
Hence, 
\[
\widehat{G_\zeta^{(m)}f}\in C^\infty_0(\R^n\setminus \Sigma_\zeta),
\]
which shows \eqref{eq_5_14}. 

Substituting $u=G_\zeta^{(m)}f$, $f\in X_\zeta$, into \eqref{eq_5_13}, and using that
\[
(-\Delta-2\zeta\cdot\nabla)^m G_\zeta^{(m)}f=f,\quad f\in X_\zeta,
\]
we get 
\begin{equation}
\label{eq_5_15}
\|G_\zeta^{(m)}f\|_{L^{\frac{2n}{n-2m}}}\le C\|f\|_{L^{\frac{2n}{n+2m}}},\quad f\in X_\zeta,
\end{equation}
where the constant $C$ is independent of $\zeta$.  According to Lemma \ref{lem_density} below the space $X_\zeta$ is dense in $L^{\frac{2n}{n+2m}}(\R^n)$, and hence the estimate \eqref{eq_5_15} can be extended to all $f\in L^{\frac{2n}{n+2m}}(\R^n)$. 
This completes the proof. 
\end{proof}

\begin{lem}
\label{lem_density}
For every $\zeta\in \C^n\setminus\{0\}$ such that $\zeta\cdot\zeta=0$,  the space 
\[
X_\zeta=\{f\in \mathcal{S}(\R^n): \hat f\in C^\infty_0(\R^n\setminus \Sigma_\zeta)\},
\]
where 
\[
\Sigma_\zeta=\{\xi\in \R^n: p_\zeta(\xi)=0\},
\]
is dense in $L^p(\R^n)$, $1< p<\infty$. 
\end{lem}

\begin{proof}
Without loss of generality, we may assume that $\zeta=s(e_1-ie_2)$, and therefore, 
\[
\Sigma_\zeta=\{\xi\in \R^n: \xi_1=0,|\xi-s e_2|=s\}.
\]
Let $g\in L^{q}(\R^n)$, $\frac{1}{p}+\frac{1}{q}=1$, be such that 
\begin{equation}
\label{eq_5_16}
\langle g, f\rangle=0
\end{equation}
 for all $f\in X_\zeta$. According to the Hahn--Banach theorem, it suffices to  show that $g=0$.  It follows from \eqref{eq_5_16} that 
 \[
\langle \mathcal{F}^{-1} g,  \hat f\rangle=0,
 \]
where $\hat f$ is an arbitrary function in $C^\infty_0(\R^n\setminus \Sigma_\zeta)$. Hence, 
\[
\supp (\mathcal{F}^{-1} g)\subset \Sigma_\zeta.
\]
As $\Sigma_\zeta$ is a compact subset contained in the hyperplane $\{\xi_1=0\}$, we have 
\[
\mathcal{F}^{-1} g=\sum_{j=0}^k u_j\otimes \p^{j} \delta(\xi_1), \quad u_j\in\mathcal{E}'(\R^{n-1}_{\xi'}),
\]
where $k$ is the order of the distribution $\mathcal{F}^{-1} g\in \mathcal{E}'(\R^n)$, see \cite[Example 5.1.2, p. 128]{Hormander_book_1}.  Thus, 
\begin{equation}
\label{eq_q_prop_density}
g=\sum_{j=0}^k \hat u_j\otimes (ix_1)^{j},\quad \hat u_j\in \mathcal{S}'\cap C^\infty(\R^{n-1}_{x'}).
\end{equation}
Since $g\in L^q(\R^n)$, $q<\infty$, by Fubini's theorem we have   $x_1\mapsto g(x_1,x')$ in $L^q(\R)$ for almost all $x'$.  As $q<\infty$, the latter is only possible if all $\hat u_j$ in \eqref{eq_q_prop_density} vanish identically. The proof is complete. 
\end{proof}

\begin{rem}
 It might be interesting to mention that there is another way to prove the density of the space $X_\zeta$ in $L^p(\R^n)$ with $\frac{2n}{n+2}\le p <\infty$, which is based  on the fact that $\Sigma_\zeta$ is a smooth manifold of codimension two, and the fact that if $g\in L^q(\R^n)$ and $\supp \hat g$ is carried by a manifold of codimension two then $g=0$ provided $1\le q\le \frac{2n}{n-2}$.  The latter fact is established in  \cite{Agranovsky_2004} by refining the proof of \cite[Theorem 7.1.27]{Hormander_book_1}.
\end{rem}

Let us finally discuss the uniqueness statement in Theorem \ref{thm_main_2}.  To that end  it remains to show that the homogeneous equation 
\begin{equation}
\label{proof_cor_1}
((-\Delta)^m-2\zeta\cdot \nabla)^m w=0
\end{equation}
has only a trivial solution in $L^{\frac{2n}{n-2m}}(\R^n)$.  Taking the Fourier transform in \eqref{proof_cor_1}, we see that $\supp(\hat w)\subset \Sigma_\zeta\subset\{\xi\in \R^n:\xi_1=0\}$. As in the proof of Lemma \ref{lem_density}, we conclude that $w=0$.  The proof of Theorem \ref{thm_main_2}  is now complete.

\section{Construction of complex geometric optics solutions}

\label{sec_cgo}

Let $q\in L^{\frac{n}{2m}}(\Omega)$ and $n>2m$. Viewing $q$ as an element of $ (L^{\frac{n}{2m}}\cap \mathcal{E}')(\R^n)$, with $\supp q\subset\overline{\Omega}$, 
consider the equation,
\begin{equation}
\label{eq_4_1}
((-\Delta)^m+q)u=0\quad \text{in}\quad  \R^n.
\end{equation}
 
The next result provides us with the existence of complex geometric optics solutions to the equation \eqref{eq_4_1}. 
\begin{prop}
\label{prop_cgo}
For each $\zeta\in \C^n$  such that $\zeta\cdot\zeta=0$ and $|\zeta|$ is sufficiently large, there exists a solution of the equation \eqref{eq_4_1} of the form
\begin{equation}
\label{eq_4_2}
u=e^{x\cdot\zeta} (1+r),
\end{equation}
where the remainder $r$ satisfies
\begin{equation}
\label{eq_4_2_remainder}
\|r\|_{L^{\frac{2n}{n-2m}}(\R^n)}=\mathcal{O}(1),\quad |\zeta|\to \infty, 
\end{equation}
and for any compact set $K\subset \R^n$, 
\[
\|r\|_{L^2(K)}\to 0,\quad |\zeta|\to \infty. 
\]
\end{prop}

\begin{proof}
We follow the method of \cite{Lavine_Nachman} and  \cite{Dos_Santos_F_Kenig_Salo_unbound}, where the existence of complex geometric optics solutions for the Schr\"odinger operator $-\Delta+q$ with $q\in L^{\frac{n}{2}}(\Omega)$ was established.  Here the convolution operator  $G_\zeta^{(m)}$ introduced in \eqref{eq_solution_op} together with the estimates \eqref{eq_5_10} and \eqref{eq_5_13_0}  will play a crucial role.

Substituting \eqref{eq_4_2} into  \eqref{eq_4_1}, we get
\begin{equation}
\label{eq_4_3}
((-\Delta-2\zeta\cdot\nabla)^m +q)r=-q\quad \text{in}\quad  \R^n.
\end{equation}
Let us write
\[
q=d_1d_2,\quad d_1=|q|^{1/2},\quad d_2=q/|q|^{1/2}.
\]
We have $d_1,d_2\in L^{\frac{n}{m}}(\R^n)$ and $\|d_j\|_{L^{\frac{n}{m}}(\R^n)}=\|q\|^{1/2}_{L^{\frac{n}{2m}}(\R^n)}$, $j=1,2$. 

We shall look for a solution of \eqref{eq_4_3} in the form 
\begin{equation}
\label{eq_4_4_-1}
r=G_\zeta^{(m)} d_1 v.
\end{equation}
Thus, we have to solve
\begin{equation}
\label{eq_4_4_0}
(I+d_2G_\zeta^{(m)} d_1)v=-d_2\in (L^{\frac{n}{m}}\cap \mathcal{E}')(\R^n)\subset (L^{2}\cap \mathcal{E}')(\R^n),
\end{equation}
as $2m<n$.  To that end  we shall invert the operator $I+d_2G_\zeta^{(m)}d_1$ on $L^2(\R^n)$. 
Let us show that 
\begin{equation}
\label{eq_norm_invert}
\|d_2G_\zeta^{(m)} d_1 \|_{L^2(\R^n)\to  L^2(\R^n)}\to 0,  \quad |\zeta|\to \infty.
\end{equation} 
In doing so, we let, when $\tau>0$,
\[
d_{j,\tau}(x)=\begin{cases} d_j(x), & |d_j(x)|\le \tau,\\
0,& \text{otherwise},
\end{cases}\quad j=1,2.
\]
Thus, for each $\tau$, $d_{j,\tau}\in L^\infty(\R^n)$. Furthermore,   $d_{j,\tau}(x)\to d_j(x)$ almost everywhere as $\tau\to \infty$.  We also have 
$|d_{j,\tau}(x)|\le |d_j(x)|$, and therefore, by  dominated convergence, we get  $\|d_j-d_{j,\tau}\|_{L^{\frac{n}{m}}(\R^n)}\to 0$ as $\tau\to\infty$. 

For $f\in L^2(\R^n)$,  we write 
\begin{equation}
\label{eq_4_5}
\begin{aligned}
\|d_2G_\zeta^{(m)} d_1 f\|_{L^2(\R^n)}\le& \|d_{2,\tau} G_\zeta^{(m)} d_{1,\tau} f \|_{L^2(\R^n)}+ \|d_{2,\tau} G_\zeta^{(m)} (d_1-d_{1,\tau}) f \|_{L^2(\R^n)}\\
&+ \|(d_2-d_{2,\tau}) G_\zeta^{(m)} d_{1} f \|_{L^2(\R^n)}.
\end{aligned}
\end{equation}

Let us now estimate each term in the right hand side of \eqref{eq_4_5}. By H\"older's  inequality and \eqref{eq_5_13_0}, for each $\tau>0$, we obtain that 
\begin{equation}
\label{eq_4_7}
\begin{aligned}
\|d_{2,\tau} G_\zeta^{(m)} (d_1-d_{1,\tau}) &f \|_{L^2(\R^n)}\le \| d_{2,\tau}\|_{L^{\frac{n}{m}}(\R^n)}\| G_\zeta^{(m)} (d_1-d_{1,\tau}) f \|_{L^{\frac{2n}{n-2m}}(\R^n)}\\
&\le C \| d_{2,\tau}\|_{L^{\frac{n}{m}}(\R^n)}
\|  (d_1-d_{1,\tau}) f \|_{L^{\frac{2n}{n+2m}}(\R^n)}\\
&\le C \| q\|_{L^{\frac{n}{2m}}(\R^n)}^{1/2}
\|  (d_1-d_{1,\tau}) \|_{L^{\frac{n}{m}}(\R^n)}\|f\|_{L^2(\R^n)}.
\end{aligned}
\end{equation}
Similarly, for each $\tau>0$, we have
\begin{equation}
\label{eq_4_8}
\begin{aligned}
\|(d_2-d_{2,\tau}) G_\zeta^{(m)} d_{1} f \|_{L^2(\R^n)}\le C\|d_2-d_{2,\tau}\|_{L^{\frac{n}{m}}(\R^n)}
\|  q \|_{L^{\frac{n}{2m}}(\R^n)}^{1/2}\|f\|_{L^2(\R^n)}.
\end{aligned}
\end{equation}
Let $\varepsilon>0$.  Since $\|d_j-d_{j,\tau}\|_{L^{\frac{n}{m}}(\R^n)}\to 0$ as $\tau\to\infty$, it follows from \eqref{eq_4_7} and \eqref{eq_4_8} that there exists $\tau$ large such that 
\begin{equation}
\label{eq_4_6_comb}
\|d_{2,\tau} G_\zeta^{(m)} (d_1-d_{1,\tau}) \|_{L^2(\R^n)\to L^2(\R^n)}+ \|(d_2-d_{2,\tau}) G_\zeta^{(m)} d_{1} \|_{L^2(\R^n)\to L^2(\R^n)}\le 2\varepsilon/3. 
\end{equation}
Let us fix this $\tau$ and obtain the estimate for the first term  in the right hand side of \eqref{eq_4_5}.  Recall that $\supp d_{j,\tau}\subset \supp q:=L$ is compact.  Letting $-m<\sigma<1-m$ and using \eqref{eq_5_10},  we get 
\begin{equation}
\label{eq_4_6}
\begin{aligned}
\|&d_{2,\tau} G_\zeta^{(m)} d_{1,\tau} f \|_{L^2(\R^n)}\le C_{L}\|d_{2,\tau} G_\zeta^{(m)} d_{1,\tau} f \|_{L^2_\sigma(\R^n)}\\
&\le C_{L}\|d_{2,\tau}\|_{L^\infty(\R^n)}\|G_\zeta^{(m)}d_{1,\tau} f \|_{L^2_\sigma(\R^n)}\\
&\le \frac{C_L}{|\zeta|^m}  \|d_{2,\tau}\|_{L^\infty(\R^n)}   \| d_{1,\tau} f\|_{L^2_{\sigma+2m-1}(\R^n)}\le 
\frac{C_L \|d_{2,\tau}\|_{L^\infty(\R^n)}  \|d_{1,\tau}\|_{L^\infty(\R^n)} }{|\zeta|^m } \|f\|_{L^2(\R^n)}.
\end{aligned}
\end{equation}
Now it follows from  \eqref{eq_4_5},   \eqref{eq_4_6_comb} and \eqref{eq_4_6} that 
\[
\|d_2G_\zeta^{(m)} d_1 \|_{L^2(\R^n)\to  L^2(\R^n)}\le \varepsilon,
\]
for  $|\zeta|$ sufficiently large, which implies \eqref{eq_norm_invert}.  

In particular, $\|d_2G_\zeta^{(m)} d_1 \|_{L^2(\R^n)\to  L^2(\R^n)}\le 1/2$ when $|\zeta|$ sufficiently large, and  therefore, \eqref{eq_4_4_0} yields  that 
\[
v=-(I+d_2G_\zeta^{(m)} d_1)^{-1}d_2=-\sum_{j=0}^{\infty} (-d_2G_\zeta^{(m)} d_1)^{j} d_2.
\]
We have 
\[
\|v\|_{L^2(\R^n)}\le 2 \|d_2\|_{L^2(\R^n)}\le \mathcal{O}(1).
\]
Using  \eqref{eq_5_13_0}, from \eqref{eq_4_4_-1}, we obtain that 
\begin{equation}
\label{eq_4_9_0}
\|r\|_{L^{\frac{2n}{n-2m}}(\R^n)}\le C\|d_1 v\|_{L^{\frac{2n}{n+2m}}(\R^n)}\le C
\|d_1 \|_{L^{\frac{n}{m}}(\R^n)}\|v\|_{L^2(\R^n)}\le \mathcal{O}(1),
\end{equation}
for  $|\zeta|$ sufficiently large.  

Let $K\subset \R^n$ be a fixed compact set and let us write 
\begin{equation}
\label{eq_4_9}
\|r\|_{L^2(K)}\le \|G_\zeta^{(m)} d_{1,\tau} v\|_{L^2(K)}+ \|G_\zeta^{(m)} (d_1-d_{1,\tau}) v\|_{L^2(K)}.
\end{equation}
Using the inclusion $L^{\frac{2n}{n-2m}}(K)\subset L^2(K)$, the estimate \eqref{eq_5_13_0} and H\"older's inequality, we get 
\begin{align*}
\|G_\zeta^{(m)}& (d_1-d_{1,\tau}) v\|_{L^2(K)} \le C_K \|G_\zeta^{(m)} (d_1-d_{1,\tau}) v\|_{L^{\frac{2n}{n-2m}}(K)}\\
&\le 
C_K \|d_1-d_{1,\tau}\|_{L^{\frac{n}{m}}(\R^n)} \|v\|_{L^2(\R^n)}\le C_{K,q} \|d_1-d_{1,\tau}\|_{L^{\frac{n}{m}}(\R^n)} . 
\end{align*}
Let $\varepsilon>0$. As $\|d_1-d_{1,\tau}\|_{L^{\frac{n}{m}}(\R^n)}\to 0$ as $\tau\to\infty$, let us fix $\tau>0$ so that 
\begin{equation}
\label{eq_4_11}
\|G_\zeta^{(m)} (d_1-d_{1,\tau}) v\|_{L^2(K)} \le \varepsilon/2.
\end{equation} 

Now let $-m<\sigma<1-m$. Using \eqref{eq_5_10} and the fact that $\supp d_1\subset \supp q:=L$ is compact,  we obtain that 
\begin{equation}
\label{eq_4_10}
\begin{aligned}
\|G_\zeta^{(m)} d_{1,\tau} v\|_{L^2(K)}&\le C_K \|G_\zeta^{(m)} d_{1,\tau} v\|_{L^2_\sigma(\R^n)}\le 
\frac{C_K }{|\zeta|^m}   \| d_{1,\tau} v\|_{L^2_{\sigma+2m-1}(\R^n)}\\
&\le \frac{C_{K,L}}{|\zeta|^m} \|d_{1,\tau}\|_{L^\infty(\R^n)} \|v\|_{L^2(\R^n)}\le \frac{C_{K,L,q}}{|\zeta|^m}\le \frac{\varepsilon}{2},
\end{aligned}
\end{equation}
for $|\zeta|$ sufficiently large.
It follows from \eqref{eq_4_9}, \eqref{eq_4_11} and \eqref{eq_4_10} that 
\[
\|r\|_{L^2(K)}\to 0, \quad |\zeta|\to \infty.
\]
The proof is complete. 
\end{proof}

\begin{rem}
\label{rem_cgo}
Let us mention that $u|_\Omega\in H^{m}(\Omega)$ where $u$ is the complex geometric optics solution  given in \eqref{eq_4_2}.  Indeed, let $\tilde \Omega\subset \R^n$ be open bounded  such that $\Omega\Subset\tilde \Omega$. 
Then it follows from \eqref{eq_4_2_remainder} that $u|_{\tilde \Omega}\in L^{\frac{2n}{n-2m}}(\tilde \Omega)$. By H\"older's inequality, we have 
\[
\|qu\|_{L^{\frac{2n}{n+2m}}(\tilde \Omega)}\le \|q\|_{L^{\frac{n}{2m}}(\tilde \Omega)}\|u\|_{L^{\frac{2n}{n-2m}}(\tilde \Omega)},
\]
 and therefore, $(-\Delta)^m u=-qu\in L^{\frac{2n}{n+2m}}(\tilde \Omega)\subset H^{-m}(\tilde \Omega)$ by the Sobolev embedding, see \cite[Theorem 0.3.7]{Sogge_book}.  Hence, by elliptic regularity, $u\in H^{m}_{\text{loc}}(\tilde \Omega)$, and thus, $u\in H^m(\Omega)$. 
\end{rem}

\section{Proof of Theorem \ref{thm_main}}

\label{sec_proof_thm_main}

An application of Lemma \ref{lem_appendix_integral_iden} and the fact that $\Lambda_{q_1}=\Lambda_{q_2}$ give us  the following integral identity,
\begin{equation}
\label{eq_6_1}
\int_\Omega (q_2-q_1) u_1\overline{u_2}dx=0,
\end{equation}
for any solutions $u_1,u_2\in H^{m}(\Omega)$ of the equations 
\begin{equation}
\label{eq_6_2}
((-\Delta)^m+q_1)u_1=0 \quad\text{in}\quad\Omega,
\end{equation}
and
\begin{equation}
\label{eq_6_3}
((-\Delta)^m+\overline{q_2})u_2=0 \quad\text{in}\quad\Omega,
\end{equation}
respectively. 

Given $\xi\in \R^n$, we set 
\begin{align*}
\zeta_1&=s\eta_1+i\bigg(\frac{\xi}{2}+r\eta_2\bigg),\\
\overline{\zeta_2}&=-s\eta_1+i\bigg(\frac{\xi}{2}-r\eta_2\bigg),
\end{align*}
where $\eta_1,\eta_2\in \mathbb{S}^{n-1}$ satisfy $\xi\cdot\eta_1=\xi\cdot\eta_2=\eta_1\cdot\eta_2=0$
 and $\frac{|\xi|^2}{4}+r^2=s^2$.
The vectors are chosen so that $\zeta_j\cdot\zeta_j=0$, $j=1,2$, and $\zeta_1+\overline{\zeta_2}=i\xi$.  We also have $|\zeta_j|=\sqrt{2}s$, $j=1,2$. 

By Proposition \ref{prop_cgo}, for $s$ sufficiently large,  there exist complex geometric optics solutions,
\begin{align*}
u_1=e^{x\cdot\zeta_1}(1+r_1),\\
u_2=e^{x\cdot\zeta_2}(1+r_2),
\end{align*}
to the equations \eqref{eq_6_2} and \eqref{eq_6_3}, respectively, where the remainders $r_j$ satisfy
\begin{equation}
\label{eq_6_4}
\|r_j\|_{L^{\frac{2n}{n-2m}}(\R^n)}=\mathcal{O}(1), \quad s\to \infty,
\end{equation}
and for any compact set $K\subset\R^n$,
\begin{equation}
\label{eq_6_5}
\|r_j\|_{L^2(K)}\to 0, \quad s\to \infty.
\end{equation}

By Remark \ref{rem_cgo} we know that  $u_1,u_2\in H^{m}(\Omega)$.  Substituting $u_1$ and $u_2$ into the integral identity \eqref{eq_6_1}, we obtain that 
\begin{equation}
\label{eq_6_6}
\int_{\Omega} (q_2-q_1)e^{i\xi\cdot x}(1+r_1+\overline{r_2}+r_1\overline{r_2})dx=0.
\end{equation}

Let us show that 
\begin{equation}
\label{eq_6_7}
\int_{\Omega} (q_2-q_1)e^{i\xi\cdot x}(r_1+\overline{r_2}+r_1\overline{r_2})dx\to 0, \quad s\to \infty.
\end{equation}
To that end, we fix $\varepsilon>0$ and write $q=q_2-q_1$. Let $q^\sharp \in L^\infty(\Omega)$ be such that $\|q-q^\sharp\|_{L^{\frac{n}{2m}}(\Omega)}\le \varepsilon$.  
By H\"older inequality, \eqref{eq_6_4} and \eqref{eq_6_5},  we get
\begin{align*}
&\bigg| \int_{\Omega} q e^{i\xi\cdot x}(r_1+\overline{r_2}+r_1\overline{r_2})dx\bigg|\\
&\le C_\Omega \|q^\sharp\|_{L^\infty(\Omega)} (\|r_1\|_{L^2(\Omega)} +\|r_2\|_{L^2(\Omega)} + \|r_1\|_{L^2(\Omega)} \|r_2\|_{L^2(\Omega)} )+ C_\Omega \|q-q^\sharp\|_{L^{\frac{n}{2m}}(\Omega)} \\
&(\|r_1\|_{L^{\frac{2n}{n-2m}}(\Omega)}+ \|r_2\|_{L^{\frac{2n}{n-2m}}(\Omega)}+ \|r_1\|_{L^{\frac{2n}{n-2m}}(\Omega)}\|r_2\|_{L^{\frac{2n}{n-2m}}(\Omega)})\le \mathcal{O}(\varepsilon),
\end{align*}
for $s$ sufficiently large, which shows \eqref{eq_6_7}. 

Taking the limit as $s\to \infty$ in \eqref{eq_6_6}, we obtain that $q_1=q_2$ in $\Omega$.  The proof of Theorem \ref{thm_main} is complete.

\begin{appendix}
\section{Wellposedness of the Dirichlet problem for $(-\Delta)^m+q$ with potential $q\in L^{\frac{n}{2m}}$}
\label{sec_appendix}

Let $\Omega\subset\R^n$ be a bounded open set with $C^\infty$ boundary, and let $q\in L^{\frac{n}{2m}}(\Omega)$, $n>2m$.  

We have the following chain of continuous inclusions, where the first and the last ones follow from the Sobolev embedding theorem,  see \cite[Theorem 0.3.7]{Sogge_book},
\[
H^m(\Omega)\hookrightarrow L^{\frac{2n}{n-2m}}(\Omega)\hookrightarrow L^2(\Omega) \hookrightarrow L^{\frac{2n}{n+2m}}(\Omega)\hookrightarrow H^{-m}(\Omega).
\]

For $f=(f_0,\dots, f_{m-1})\in\prod_{j=0}^{m-1} H^{m-j -1/2}(\p \Omega)$, consider the following Dirichlet problem, 
\begin{equation}
\label{eq_7_1}
\begin{aligned}
((-\Delta)^m+q)u&=0\quad\textrm{in}\quad \Omega,\\
\gamma u&=f \quad\textrm{on}\quad \p\Omega.
\end{aligned}
\end{equation}
Here 
\[
\gamma:H^{m}(\Omega)\to  \prod_{j=0}^{m-1}H^{m-j-1/2}(\p \Omega),\quad \gamma u=(u|_{\p\Omega},\p_{\nu}u|_{\p \Omega},\dots,\p_{\nu}^{m-1}u|_{\p \Omega})
\]
is the Dirichlet trace of $u\in H^m(\Omega)$ on the boundary of $\Omega$, and $\nu$ is the exterior unit normal to the boundary. 

The purpose of this appendix is to use the standard variational arguments  to show the wellposedness of the problem \eqref{eq_7_1}. 
Consider first the inhomogeneous problem, 
\begin{equation}
\label{eq_7_1_inhom}
\begin{aligned}
((-\Delta)^m+q)u&=F\quad\textrm{in}\quad \Omega,\\
\gamma u&=0 \quad\textrm{on}\quad \p\Omega,
\end{aligned}
\end{equation}
with $F\in H^{-m}(\Omega)$.
Using the multinomial theorem, we write
\[
(-\Delta)^m=\sum_{|\alpha|=m} \frac{m!}{\alpha!} D^{2\alpha}.
\]
To define a natural  sesquilinear form $a$, associated to the problem \eqref{eq_7_1_inhom}, we let $u,v\in C^\infty_0(\Omega)$ and integrate by parts, 
\[
((-\Delta)^m+q) u, v)_{L^2(\Omega)}=  \sum_{|\alpha|=m}  \frac{m!}{\alpha!} \int_\Omega  D^\alpha u \overline{D^\alpha v} dx + \int_\Omega u q\overline{v} dx:=a(u,v).
\]
Notice that this is not a unique form, associate with the problem  \eqref{eq_7_1_inhom}. 

Using the Sobolev embedding $H^m(\Omega)\subset L^{\frac{2n}{n-2m}}(\Omega)$ and H\"older's inequality, we obtain that 
\begin{equation}
\label{eq_7_2_bounded}
\begin{aligned}
|a(u,v)|&\le \sum_{|\alpha|=m}  \frac{m!}{\alpha!} \| D^\alpha u\|_{L^2(\Omega)}  \| D^\alpha v\|_{L^2(\Omega)} + \|q\|_{L^\frac{n}{2m}(\Omega)}\|u\|_{L^\frac{2n}{n-2m}(\Omega)} \|v\|_{L^\frac{2n}{n-2m}(\Omega)}
\\
&\le C\|u\|_{H^m(\Omega)}\|v\|_{H^m(\Omega)}.
\end{aligned}
\end{equation}
Hence,  the sesquilinear form $a(u,v)$ extends to a bounded form on $H^m_0(\Omega)$.

Poincar\'e's inequality implies that for $|\beta|<m$, we have 
\[
\|D^\beta u\|_{L^2(\Omega)}\le C\sum_{|\alpha|=m}\|D^\alpha u\|_{L^2(\Omega)}, \quad u\in H^m_0(\Omega),
\]
and therefore, 
\begin{equation}
\label{eq_7_2}
\|u\|_{H^m(\Omega)}^2\le C\sum_{|\alpha|=m}\|D^\alpha u\|_{L^2(\Omega)}^2, \quad u\in H^m_0(\Omega).
\end{equation}
Using \eqref{eq_7_2}, and
writing $q=q^\sharp+(q-q^\sharp)$ with $q^\sharp\in L^\infty(\Omega)$ and $\|q-q^\sharp\|_{L^{\frac{n}{2m}}(\Omega)}$ small enough, we obtain that 
\begin{align*}
\text{Re}\, a(u,u)&\ge \sum_{|\alpha|=m}  \frac{m!}{\alpha!} \int_\Omega  |D^\alpha u|^2 dx - \int_\Omega |u|^2 |q| dx\\
&\ge c   \sum_{|\alpha|=m} \|D^\alpha u\|^2_{L^2(\Omega)} -\|q^\sharp\|_{L^\infty(\Omega)} \|u\|^2_{L^2(\Omega)}- \|q-q^\sharp\|_{L^{\frac{n}{2m}}(\Omega)} \|u\|_{L^\frac{2n}{n-2m}(\Omega)}^2\\
&\ge  (c/2)\|u\|_{H^m(\Omega)}^2 - C_0\|u\|^2_{L^2(\Omega)}, \quad c>0,\quad u\in H^m_0(\Omega). 
\end{align*}
Thus, the form $a(u,v)$ is coercive on $H^m_0(\Omega)$. As the inclusion map $H^m_0(\Omega)\hookrightarrow L^2(\Omega)$ is compact, the operator 
\begin{equation}
\label{eq_7_3}
(-\Delta)^m+q: H^m_0(\Omega)\to H^{-m}(\Omega)=(H^{m}_0(\Omega))'
\end{equation}
is Fredholm of index zero, see \cite[Theorem 2.34]{McLean_book}.

Furthermore, since the operator $(-\Delta)^m+q +C_0: H^m_0(\Omega)\to H^{-m}(\Omega)$ is positive, by an application of the Lax--Milgram lemma we conclude that it has a bounded inverse. As the embedding $H_0^m(\Omega)\hookrightarrow H^{-m}(\Omega)$ is compact, the operator \eqref{eq_7_3}, viewed as an operator on the Hilbert space $H^{-m}(\Omega)$, has a discrete spectrum. 

To study the well-posedness of \eqref{eq_7_1}, let us assume that 
\begin{itemize}

\item[(A)] $0$ is not in the spectrum of the operator \eqref{eq_7_3}.
\end{itemize}
Let $w\in H^m(\Omega)$ be such that $\gamma w=f$, see  \cite[Theorem 9.5, p. 226]{Grubbbook2009} for the existence of such $w$. Then   $u=v+w\in H^m(\Omega)$, where  $v\in H^m_0(\Omega)$ is the unique solution of the equation, 
\[
((-\Delta)^m+q) v=-((-\Delta)^m+q)w\in H^{-m}(\Omega),
\]
solves the Dirichlet problem \eqref{eq_7_1}.  Furthermore, the solution to the Dirichlet problem \eqref{eq_7_1} is unique.

Under the assumption (A), we define the Dirichlet--to--Neumann map, associated to \eqref{eq_7_1}, in the following way. Let $f, h\in \prod_{j=0}^{m-1}H^{m-j-1/2}(\p \Omega)$, and $v\in H^m(\Omega)$ be such that $\gamma v=h$. Then we set 
\begin{equation}
\label{eq_7_4}
\langle \Lambda_q f, \overline{h} \rangle=a(u,v)=\sum_{|\alpha|=m}  \frac{m!}{\alpha!} \int_\Omega  D^\alpha u \overline{D^\alpha v} dx + \int_\Omega u q\overline{v} dx,
\end{equation}
where $u\in H^m(\Omega)$ is the unique solution of the Dirichlet problem \eqref{eq_7_1}.

Let us now show that the definition \eqref{eq_7_4} of $\Lambda_q f$ is independent of the choice of an extension $v$ of $h$. To that end let $v_1,v_2\in H^m(\Omega)$ be such that $\gamma v_1=\gamma v_2=h$. Then we have to show that 
\begin{equation}
\label{eq_7_5}
\sum_{|\alpha|=m}  \frac{m!}{\alpha!} \int_\Omega  D^\alpha u \overline{D^\alpha (v_1- v_2)}dx + \int_\Omega u q(\overline{v_1}-\overline{v_2}) dx=0. 
\end{equation}
For any $w\in C^\infty_0(\Omega)$, as $u\in H^m(\Omega)$, we have 
\[
0=\langle ((-\Delta)^m+q)u, \overline{w} \rangle=\sum_{|\alpha|=m}  \frac{m!}{\alpha!}(-1)^m \int_\Omega  D^\alpha u D^\alpha \overline{w} dx + \int_\Omega u q \overline{w}dx.
\]
As $C_0^\infty(\Omega)$ is dense in $H^m_0(\Omega)$ and the form is continuous on $H^m_0(\Omega)$, we get \eqref{eq_7_5}.

It follows from \eqref{eq_7_2_bounded} that 
\[
|\langle \Lambda_q f, \overline{h} \rangle|\le C\|u\|_{H^m(\Omega)}\|v\|_{H^m(\Omega)}\le C\|f\|_{\prod_{j=0}^{m-1}H^{m-j-1/2}(\p \Omega)}\|h\|_{\prod_{j=0}^{m-1}H^{m-j-1/2}(\p \Omega)},
\]
where 
\[
\|h\|_{\prod_{j=0}^{m-1}H^{m-j-1/2}(\p \Omega)}=(\|h_0\|_{H^{m-1/2}(\p\Omega)}^2+\cdots+\|h_{m-1}\|_{H^{1/2}(\p\Omega)}^2)^{1/2}
\]
is the product norm on the space $\prod_{j=0}^{m-1}H^{m-j-1/2}(\p \Omega)$. 
Here we have used the fact that the extension operator $\prod_{j=0}^{m-1}H^{m-j-1/2}(\p \Omega)\ni h\mapsto v\in H^m(\Omega)$ is bounded, see \cite[Theorem 9.5, p. 226]{Grubbbook2009}.  
Hence, 
\[
\Lambda_q f\in \bigg(\prod_{j=0}^{m-1}H^{m-j-1/2}(\p \Omega)\bigg)'= \prod_{j=0}^{m-1}H^{-m+j+1/2}(\p \Omega)
\]
is well defined, and the operator 
\[
\Lambda_q: \prod_{j=0}^{m-1}H^{m-j-1/2}(\p \Omega) \to \prod_{j=0}^{m-1}H^{-m+j+1/2}(\p \Omega)
\]
is bounded.

The following integral identity is used in the proof of Theorem \ref{thm_main}. 
\begin{lem}
\label{lem_appendix_integral_iden}
Let $q_1,q_2\in L^{\frac{n}{2m}}(\Omega)$ and $\Lambda_{q_1}=\Lambda_{q_2}$. Then 
\begin{equation}
\label{eq_7_6_lem}
\int_{\Omega} (q_2-q_1)u_1 \overline{u_2}dx=0,
\end{equation}
for any solutions $u_1,u_2\in H^{m}(\Omega)$ of the equations $((-\Delta)^m+q_1)u_1=0$ in $\Omega$,  $((-\Delta)^m+\overline{q_2})u_2=0$ in $\Omega$, respectively. 
\end{lem}

\begin{proof}
First as $u_2\in H^{m}(\Omega)$ satisfies the equation $((-\Delta)^m+q_2)\overline{u_2}=0$,     we have
\begin{equation}
\label{eq_7_6}
0=\langle \overline{u_2} , ((-\Delta)^m+q_2)\varphi \rangle=
\sum_{|\alpha|=m}  \frac{m!}{\alpha!}(-1)^m \int_\Omega  D^\alpha \overline{u_2} D^\alpha \varphi dx + \int_\Omega \overline{u_2} q_2 \varphi dx,
\end{equation}
for any $\varphi\in C_0^\infty(\Omega)$.  By density and continuity, \eqref{eq_7_6} remains valid for any $\varphi\in H_0^m(\Omega)$. 

Let $v_2\in H^{m}(\Omega)$ be such that 
\begin{align*}
((-\Delta)^m+q_2)v_2&=0\quad \text{in}\quad \Omega,\\
\gamma v_2&=\gamma u_1.
\end{align*}
Substituting $\varphi=u_1-v_2\in H^m_0(\Omega)$ into the identity \eqref{eq_7_6}, we get 
\begin{equation}
\label{eq_7_7}
\sum_{|\alpha|=m}  \frac{m!}{\alpha!} \int_\Omega  \overline{D^\alpha u_2} D^\alpha (u_1-v_2) dx + \int_\Omega \overline{u_2} q_2 (u_1-v_2)dx=0.
\end{equation}
From the equality $\langle \Lambda_{q_1}(\gamma u_1) , \overline{\gamma u_2} \rangle= \langle \Lambda_{q_2}(\gamma v_2) , \overline{\gamma u_2} \rangle$ we conclude that
\begin{equation}
\label{eq_7_8}
\sum_{|\alpha|=m}  \frac{m!}{\alpha!} \int_\Omega  \overline{D^\alpha u_2} D^\alpha (u_1-v_2) dx +
 \int_\Omega  (u_1q_1- v_2q_2)\overline{u_2} dx=0.
\end{equation}
Subtracting \eqref{eq_7_8} from \eqref{eq_7_7}, we obtain \eqref{eq_7_6_lem}. The proof is complete. 
\end{proof}
 
\end{appendix}

\section*{Acknowledgements}

The research of K.K. is partially supported by the
Academy of Finland (project 255580) and the National Science Foundation (DMS 1500703). K.K. would like to thank Alberto Ruiz  for some very helpful discussions.  The research of
G.U. is partially supported by the National Science Foundation.

\end{document}